\documentclass[final,3p,times,sort&compress]{elsarticle}
\usepackage{lineno,hyperref}
\usepackage{caption}
\usepackage{subcaption}
\usepackage{framed}
\usepackage{amsmath}
\usepackage{bm}
\usepackage{dsfont}
\usepackage{graphicx,color,url}
\usepackage{epstopdf}
\usepackage{multirow}
\usepackage{chngpage}
\usepackage{array}
\usepackage{float}
\usepackage{lineno}
\usepackage[utf8]{inputenc}
\usepackage{setspace} 

\usepackage[english]{babel}
\usepackage{amsthm}
\usepackage{epstopdf}
\usepackage{mathtools,xparse}
\usepackage{nomencl}
\usepackage{ifthen}

\usepackage{amsfonts}
\usepackage{indentfirst}

\renewcommand\nomgroup[1]{%
	\ifthenelse{\equal{#1}{A}}{%
		\item[\textbf{Dimensions}]}{%
	\ifthenelse{\equal{#1}{B}}{%
		\item[\textbf{Dynamics}]}{%
	\ifthenelse{\equal{#1}{S}}{%
		\item[\textbf{Sensitivity Analysis}]}{%
	\ifthenelse{\equal{#1}{G}}{%
		\item[\textbf{General}]}{%
	\ifthenelse{\equal{#1}{U}}{%
		\item[\textbf{Subscripts}]}{%
	\ifthenelse{\equal{#1}{X}}{%
		\item[\textbf{Other Symbols}]}{%
	{}}}}}}}}

\makenomenclature

\modulolinenumbers[5]

\newtheorem{theorem}{Theorem}
\newdefinition{definition}{Definition}
\newdefinition{remark}{Remark}
\newdefinition{assumption}{Assumption}
\newproof{pf}{Proof}

\newcommand{\bq}{{q}}

\newcommand{\brho}{{\rho}}

\newcommand{\blambda}{{\mu}}
\newcommand{\bv}{{v}}
\newcommand{\bV}{{V}}
\newcommand{\bZ}{{ Z}}
\newcommand{\bx}{{x}}
\newcommand{\bX}{{X}}
\newcommand{\bz}{{z}}

\newcommand{\bPhi}{{ \Phi}}
\newcommand{\timp}{{t_{\rm eve}}}
\newcommand{\beforeimp}{|_{\timp}^{-}}
\newcommand{\afterimp}{|_{\timp}^{+}}
\newcommand{\atimp}{|_{\timp}}
\newcommand{\nc}{{n_c}}

\newcommand{\bg}{{g}}
\newcommand{\fun}{{ f}}

\newcommand{\Force}{\mathsf{F}}

\newcommand{\Faccel}{\mathsf{C}}
\newcommand{\bQ}{{ Q}}
\newcommand{\Mass}{\mathsf{M}}

\newcommand{\Rez}{\mathsf{R}}

\newcommand{\bzero}{\mathsf{0}}
\newcommand{\Zero}{\mathsf{0}}

\newcommand{\bI}{\mathsf{I}}

\newcommand{\Permutation}{\mathsf{P}}

\newcommand{\fin}{{\left(\,t,\,\bq,\,\bv,\,\brho \, \right)}}

\newcommand{\fdaeimpv}{{ \fun^{\scalebox{0.5}{\rm DAE-imp-$\bv$}}}}  
\newcommand{\fdaeimpl}{{ \fun^{\scalebox{0.5}{\rm DAE-imp-$\mu$}}}}  

\newcommand{\fdaeimplbq}{{ \fun_\bq^{\scalebox{0.5}{\rm DAE-imp-$\mu$}}}}
\newcommand{\fdaeimplbv}{{ \fun_\bv^{\scalebox{0.5}{\rm DAE-imp-$\mu$}}}}
\newcommand{\fdaeimplbrho}{{ \fun_\brho^{\scalebox{0.5}{\rm DAE-imp-$\mu$}}}}  

\newcommand{\fdaeimpvbq}{{ \fun_\bq^{\scalebox{0.5}{\rm DAE-imp-$\bv$}}}}
\newcommand{\fdaeimpvbv}{{ \fun_\bv^{\scalebox{0.5}{\rm DAE-imp-$\bv$}}}}
\newcommand{\fdaeimpvbrho}{{ \fun_\brho^{\scalebox{0.5}{\rm DAE-imp-$\bv$}}}} 

\newcommand{\feom}{{ \fun^{\scalebox{0.5}{\rm eom}}}}

\newcommand{\feomq}{{ \fun_\bq^{\scalebox{0.5}{\rm eom}}}}
\newcommand{\feomv}{{ \fun_\bv^{\scalebox{0.5}{\rm eom}}}}
\newcommand{\feomzeta}{{ \fun_\zeta^{\scalebox{0.5}{\rm eom}}}}
\newcommand{\feomrho}{{\fun_{\brho}^{\scalebox{0.5}{\rm eom}}}}
\newcommand{\feomrhoi}{{\fun_{\brho_i}^{\scalebox{0.5}{\rm eom}}}}
\newcommand{\fdaeq}{{ \fun_\bq^{\scalebox{0.4}{\rm DAE}}}}
\newcommand{\fdaev}{{ \fun_\bv^{\scalebox{0.4}{\rm DAE}}}}
\newcommand{\fdaelb}{{ \fun^{\scalebox{0.4}{\rm DAE-{\scalebox{1.5}{$\mu$}}}}}}
\newcommand{\fdaelbq}{{ \fun_\bq^{\scalebox{0.4}{\rm DAE-{\scalebox{1.5}{$\mu$}}}}}}
\newcommand{\fdaelbv}{{ \fun_\bv^{\scalebox{0.4}{\rm DAE-{\scalebox{1.5}{$\mu$}}}}}}

\newcommand{\fdaelbrho}{{ \fun_{\brho}^{\scalebox{0.4}{\rm DAE-{\scalebox{1.5}{$\mu$}}}}}}
\newcommand{\fdaedv}{{ \fun^{\scalebox{0.4}{\rm DAE-{\scalebox{1.5}{$\dbv$}}}}}}
\newcommand{\fdaedvq}{{ \fun_\bq^{\scalebox{0.4}{\rm DAE-{\scalebox{1.5}{$\dbv$}}}}}}
\newcommand{\fdaedvv}{{ \fun_\bv^{\scalebox{0.4}{\rm DAE-{\scalebox{1.5}{$\dbv$}}}}}}

\newcommand{\fdaedvrho}{{ \fun_{\brho}^{\scalebox{0.4}{\rm DAE-{\scalebox{1.5}{$\dbv$}}}}}}
\newcommand{\fdaerho}{{ \fun_\brho^{\scalebox{0.4}{\rm DAE}}}}

\newcommand{\Z}{\bZ \fin}

\newcommand{\bu}{u}
\renewcommand{\Re}{{\mathds{R}}}

\newcommand{\dPhi}{{\dot \bPhi}}
\newcommand{\ddPhi}{{\ddot \bPhi}}
\newcommand{\dPhidq}{{\bPhi}_{\bq}}
\newcommand{\dPhidqq}{{\bPhi}_{\bq , \, \bq}}
\newcommand{\dPhidqdrho}{\bPhi_{\bq, \, \brho}}
\newcommand{\dPhidtdq}{\bPhi_{t, \, \bq }}

\newcommand{\dPhidqplus}{{{\bPhi}^+_{\bq}}}
\newcommand{\dPhidqqplus}{{{\bPhi}^+_{\bq , \, \bq}}}
\newcommand{\dPhidtdrhoplus}{{\bPhi^+_{t,\, \brho }}}
\newcommand{\dPhidqdrhoplus}{{\bPhi^+_{\bq, \, \brho}}}
\newcommand{\dPhidtdqplus}{{\bPhi^+_{t, \, \bq }}}
\newcommand{\dPhidtdvplus}{{\bPhi^+_{t, \,\bv }}}

\newcommand{\dtdPhidq}{{\dPhi}_{\bq}}

\newcommand{\frdtdrho}{\frac{d\timp}{d\brho}}

\newcommand{\dtdrho}{{d\timp}/{d\brho}}

\newcommand{\dbq}{{\dot \bq}} 
 
\newcommand{\ddbq}{{\ddot \bq}} 

\newcommand{\dbv}{{\dot \bv}}
\newcommand{\dbZ}{{\dot \bZ}}
\newcommand{\dbz}{{\dot \bz}}

\newcommand{\dbV}{{\dot{\bV}}}

\newcommand{\qrhoplus}{{\bQ \afterimp}}
\newcommand{\qrhominus}{{\bQ \beforeimp}}

\newcommand{\dqrhoplus}{{\bV \afterimp}}
\newcommand{\dqrhominus}{{\bV \beforeimp}}

\newcommand{\Zplus}{{\bZ \afterimp}}
\newcommand{\Zminus}{{\bZ \beforeimp}}

\newcommand{\timpplus}{|_{\timp}^{+}}
\newcommand{\timpminus}{|_{\timp}^{-}}

\newcommand{\gplus}{\bg \timpplus}
\newcommand{\gminus}{\bg \timpminus}
\newcommand{\gplusarg}{\tilde{g} \big(\timp,\qtimp,\vplus, \brho\big)}
\newcommand{\gminusarg}{\tilde{g} \big(\timp,\qtimp,\vminus, \brho\big)}

\newcommand{\vplus}{{\bv \timpplus}}
\newcommand{\vminus}{\bv \timpminus}
\newcommand{\qplus}{{\bq \timpplus}}
\newcommand{\qminus}{\bq \timpminus}
\newcommand{\qtimp}{\bq|_{\timp}}
\newcommand{\vtimp}{\bv|_{\timp}}
\newcommand{\zplus}{{z \timpplus}}
\newcommand{\zminus}{z \timpminus}

\newcommand{\eps}{\varepsilon}

\newcommand{\zddytwo}{\bz(t) = \int_{t_0}^{t} \ddot y_2(\tau) \ {\rm d\tau}}
\newcommand{\dytwo}{\int_{t_0}^{t} \dot y_2(\tau) \ {\rm d\tau}}
\newcommand{\ddytwo}{\int_{t_0}^{t} \ddot y_2 \ {\rm d\tau}}

\nomenclature[Ba]{$\feom$}{The function solving the Equation of Motion  $\in \mathds{R} \times \mathds{R}^{n} \times \mathds{R}^{n} \times \mathds{R}^{p} \rightarrow \mathds{R}^{n}$ }

\nomenclature[Bc]{$\bq , \, \dbq$ $\in \mathds{R}^{n} $}{The generalized position and velocity state vector}%
\nomenclature[Be]{$\ddbq$ $\in \mathds{R}^{n} $ }{The generalized acceleration state vector}%
\nomenclature[Bf]{$z$ $\in \mathds{R}^{\nc} $}{The vector of quadrature variables}%

\nomenclature[Bg]{$\bx$ $\in \mathds{R}^{(2n+p+n_c)}$}{The state vector of the canonical ODE}

\nomenclature[Bh]{$\timp$ $\in \mathds{R}$}{The time of event}%
\nomenclature[Bi]{$\brho$ $\in \mathds{R}^{p} $}{ The vector of system parameters}%

\nomenclature[Bj]{$\Force$}{The generalized force vector $\in \mathds{R} \times \mathds{R}^{n} \times \mathds{R}^{n} \times \mathds{R}^{p} \rightarrow \mathds{R}^{n}$}%
\nomenclature[Bk]{$\Mass$}{The generalized smooth and invertible Mass matrix $\in \mathds{R} \times \mathds{R}^{n} \times \mathds{R}^{p} \rightarrow \mathds{R}^{n \times n}$}%
\nomenclature[Bl]{$\bPhi$}{The equations of constraints $\in \mathds{R}  \times \mathds{R}^{n} \times \mathds{R}^{p} \rightarrow \mathds{R}^{m}$}%

\nomenclature[G]{$\square_{\zeta, \phi}$}{ Double subscripts indicates a three-dimensional Jacobian with respect to a quantity $\zeta$ and $\phi$, unless stated otherwise}
\nomenclature[G]{$\square_{\zeta}$}{ Subscript indicates partial derivative with respect to a quantity $\zeta$, unless stated otherwise}
\nomenclature[G]{$\dot{\square}$ or $\ddot{\square}$}{The total (first or second order) derivative of a function or variable with respect to time}

\nomenclature[S]{$\bQ$ $\in \mathds{R}^{n \times p}$}{The sensitivity matrix of the state vector $\bq$ with respect to the vector of system parameters $\brho$}%
\nomenclature[S]{$\bV$ $\in \mathds{R}^{n \times p}$}{The sensitivity matrix of the state vector $\dbq$ with respect to the vector of system parameters $\brho$ }%
\nomenclature[S]{$Z$ $\in \mathds{R}^{\nc \times p}$}{The sensitivity matrix of the vector of quadrature variables $\bz$ }%
\nomenclature[S]{$\lambda^{\bQ}$ $\in \mathds{R}^{n \times n_c}$}{The adjoint sensitivity matrix of $\bQ$}
\nomenclature[S]{$\lambda^{\bV}$$\in \mathds{R}^{n \times n_c}$}{The adjoint sensitivity matrix of $\bV$ }
\nomenclature[S]{$\lambda$$\in \mathds{R}^{(2n+p+n_c) \times n_c}$}{The adjoint sensitivity matrix of $\bX$ }%
\nomenclature[S]{$\bX$ $\in \mathds{R}^{(2n+p+n_c)\times p}$ }{The sensitivity matrix of the $\bx$ state vector with respect to the vector of system parameters $\brho$ }
\nomenclature[S]{$\dtdrho$ $\in \mathds{R}^{1 \times p} $ }{The sensitivity of the time of event $\timp$ with respect to the vector of system parameters $\brho$ }%
\nomenclature[S]{$\psi$ $\in \mathds{R}^{\nc} $ }{The vector of cost functions}%
\nomenclature[S]{$g$ $\in \mathds{R}^{\nc} $}{The vector of trajectory cost functions  }%
\nomenclature[S]{$w$ $\in \mathds{R}^{\nc} $}{The vector of terminal cost functions  }%

\nomenclature[Aa]{$n$}{The number of generalized coordinates}
\nomenclature[Ab]{$p$}{The number of parameters}
\nomenclature[Ad]{$m$}{The number of equations of constraints}
\nomenclature[Ac]{$\nc$}{The number of cost functions }

\journal{Nonlinear Analysis: Hybrid Systems}

\begin{document}
\thispagestyle{empty}
\setcounter{page}{0}

\makeatletter
\def\Year#1{%
  \def\yy@##1##2##3##4;{##3##4}%
  \expandafter\yy@#1;
}
\makeatother

\begin{Huge}
\begin{center}
Computational Science Laboratory Technical Report CSL-TR-\Year{\the\year}-{\tt 3} \\
\today
\end{center}
\end{Huge}
\vfil
\begin{huge}
\begin{center}
Sebastien Corner, Corina Sandu, Adrian Sandu
\end{center}
\end{huge}

\vfil
\begin{huge}
\begin{it}
\begin{center}
``{\tt Adjoint Sensitivity Analysis of Hybrid  Multibody Dynamical Systems}''
\end{center}
\end{it}
\end{huge}
\vfil

\begin{large}
\begin{center}
Computational Science Laboratory \\
Computer Science Department \\
Virginia Polytechnic Institute and State University \\
Blacksburg, VA 24060 \\
Phone: (540)-231-2193 \\
Fax: (540)-231-6075 \\ 
Email: \url{scorner@vt.edu} \\
Web: \url{http://csl.cs.vt.edu}
\end{center}
\end{large}

\vspace*{1cm}

\begin{tabular}{ccc}
\includegraphics[width=2.5in]{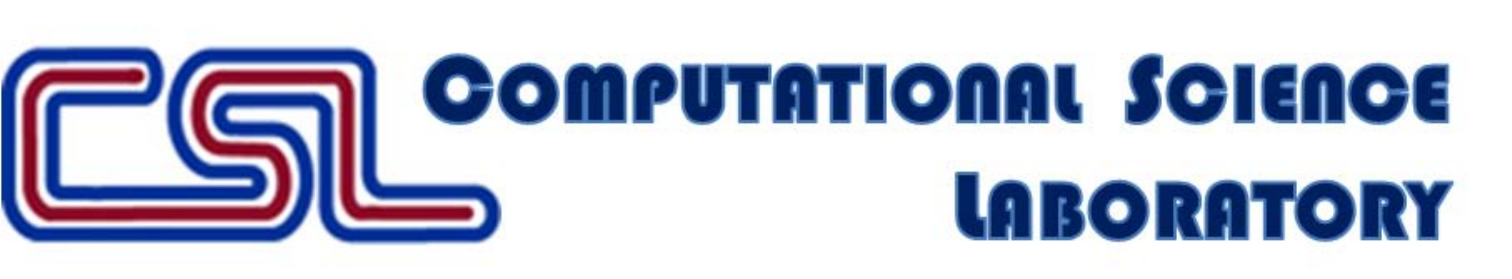}
&\hspace{2.5in}&
\includegraphics[width=2.5in]{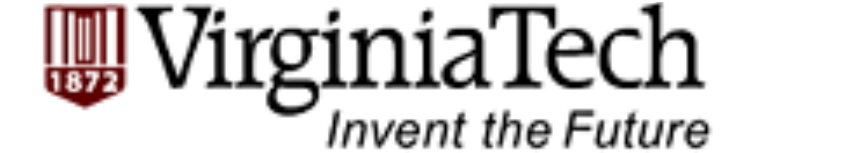} \\
{\bf\large \textit{Compute the Future}} &&\\
\end{tabular}

\newpage

\begin{frontmatter}

\title{Adjoint Sensitivity Analysis of Hybrid  Multibody Dynamical Systems} 


\author[rvt,focal]{Sebastien Corner}
\ead{scorner@vt.edu}
\author[focal]{Adrian Sandu}
\ead{sandu@cs.vt.edu}
\author[rvt]{Corina Sandu}
\ead{csandu@vt.edu}

\address[rvt]{Department of Mechanical Engineering,
Virginia Tech,
Blacksburg, VA 24061}
\address[focal]{Computational Science Laboratory,
Department of Computer Science,
Virginia Tech,
Blacksburg, VA 24061}

\begin{abstract}
Sensitivity analysis of multibody systems computes the derivatives of general cost functions that depend on the system solution with respect to parameters or initial conditions. 
This work develops adjoint sensitivity analysis for hybrid multibody dynamic systems. Hybrid systems are characterized by trajectories that are piecewise continuous in time, with finitely-many discontinuities being caused by events such as elastic/inelastic impacts or sudden changes in constraints. The corresponding direct and adjoint sensitivity variables are also discontinuous at the time of events. The framework discussed herein uses a jump sensitivity matrix to relate the jump conditions for the direct and adjoint sensitivities before and after the time event, and provides analytical jump equations for the adjoint variables. 
The theoretical framework for sensitivities for hybrid systems is validated on a five-bar mechanism with non-smooth contacts.
\end{abstract}
\end{frontmatter}


\doublespacing

\section{Introduction}

Hybrid dynamical systems are characterized piecewise-in-time smooth trajectories, with discontinuities appearing at a finite number of time moments as a consequence of external events. The discontinuities are characterized by a jump in the generalized velocity variables, e.g., caused by an impact or/and an abrupt change on the right hand side of the equation of motion. 

Sensitivity analysis aims to quantify the effect of small changes in the system parameters (inputs) on a general cost function (outputs) ~\cite{Sandu2015dynamic}. Sensitivity analysis is essential in solving computational engineering problems such as design and control optimization, implicit time integration methods, and deep learning. Finite difference methods that approximate the sensitivities by the difference between perturbed and nominal trajectories are often inaccurate ~\cite{chang2016design}. Two complementary approaches to sensitivity analysis are widely used, the direct and adjoint methods. While they provide the same derivatives, their approach and computational costs are different. Direct sensitivity propagates small perturbations forward through the system dynamics, while the adjoint method performs an inverse modeling that attempts to identify the origin of uncertainty in the model output ~\cite{sandu2003direct}. 

The sensitivity analysis with respect to system parameters and initial conditions for hybrid systems has been studied in the literature ~\cite{Barton1998,Barton1999, barton2002modeling, tolsma2002hidden, rozenvasser1967general, saccon2014sensitivity, hiskens2006sensitivity, hiskens2000trajectory, taringoo2010geometry, backer1966jump}. The jump conditions of the direct sensitivities for hybrid ODE systems were first presented by Becker \cite{backer1966jump} in 1966 and a year latter by Rozenvasser ~\cite{rozenvasser1967general}. Thirty years later, Gal\'{a}n {\it{el al.}} \cite{Barton1999} presented sufficient conditions for the existence and uniqueness of these jump equations.  Jump conditions involve the sensitivity of the time of event, and the jumps in the sensitivities of the state variables at the time of event. Within the same period, Hiskens applied this theory to power switching systems \cite{hiskens2000trajectory}. The jump conditions of the adjoint sensitivities for hybrid ODE systems with discontinuities in the right-hand side and with switching manifold parameters  were presented by Stewart \cite{stewart2010optimal} and Taringoo \cite{Taringoo2009}, respectively. Recently, Zhang et al. \cite{Hong2017} derived the jump conditions for adjoints of differential-algebraic systems and applied them to large-scale power systems with switching dynamics. 

This paper provides a unified mathematical framework for the direct and adjoint sensitivity analysis for multibody dynamic systems and general cost functions. The framework includes both unconstrained and constrained mechanical systems. The direct sensitivity analysis was developed in \cite{Corner2017}, where a new graphical proof of the jump conditions for direct sensitivity variables was given. Jump conditions for constrained mechanical systems with change of mechanism and dealing with impulsive forces at the time of event was also presented. 

This paper extends the mathematical framework to perform adjoint sensitivity analysis for mechanical systems with non-smooth trajectories. The approach taken here is as follows. An event detection mechanism (e.g., embedded in the numerical integration method) finds the time of the next event. At this time moment the trajectories of the generalized position variables are continuous but not differentiable, while the trajectories of the generalized velocities are discontinuous due to either impulsive external forces or to abrupt changes of the right-hand side.  The jump conditions that map the direct sensitives from right before the event to right after the event can be formalized with the help of a jump sensitivity matrix.  The jumps in the adjoint sensitivity variables are obtained via the transpose of this jump sensitivity matrix.

The paper is organized as follows. A review of the direct and adjoint sensitivity analyses for smooth dynamical systems with extended general cost functions is given in Section \ref{sec:sensitivity_unconstrained_systems}. The framework for direct and adjoint sensitivity analyses for hybrid  dynamical systems is discussed in Section \ref{sec:sensitivity_constrained_systems}. The methodology is applied to compute sensitivities of a five-bar mechanism with non-smooth contacts in Section \ref{sec:numerics}. Conclusions are drawn in Section \ref{sec:conclusions}.

\section{Sensitivity analysis for unconstrained mechanical systems and extended cost functions}
\label{sec:sensitivity_unconstrained_systems}

This section provides a summary of a previously method developed by the authors to implement direct sensitivity analysis for dynamical systems governed by smooth second order systems of ordinary differential equations (2nd order ODEs). More details of this method can be found here (cite the paper). This method is extended to multiple cost functions that contain argument function.
%
\subsection{Smooth ODE system dynamics and extended cost functions}
%
We consider an unconstrained mechanical system governed by the second order ordinary differential equation (ODE):
\begin{equation}
\label{eq:EOM-ODE}
\begin{split}
{\Mass}\left(t,\bq,\brho\right) \cdot \ddbq &={\Force} \left(t,\bq,\dbq,\brho\right),
\quad t_0 \leq t \leq  t_F, \quad \bq(t_0)=\bq_0(\brho),\quad\dbq(t_0)=\dbq_0(\brho), \\
\Leftrightarrow ~~~\ddbq &={\Mass}^{-1} \left(t,\bq,\brho\right)  \cdot {\Force}\left(t,\bq,\dbq,\brho\right)
=: \feom \left(t,\bq,\dbq,\brho\right),
\end{split}
\end{equation}
where $\bq \in \Re^n$ are the generalized positions, $\bv := \dbq \in \Re^n$ the generalized velocities, and $\brho \in \Re^p$ the time independent parameters of the system. The state trajectories depend implicitly on time and on the parameters, $\bq=\bq(t, \brho)$ and $\bv=\dbq(t, \brho)$. We consider a general system output of the form:
\begin{eqnarray}
\label{eq:CostFunction}
\psi(\rho) =  \int_{t_0}^{t_F} {\tilde{g}_j\bigl(\tau,\bq,\bv,\brho\bigr) \  d\tau} + 
\tilde{w}\bigl(t_F,\bq_{t_F},\bv_{t_F},\brho\bigr) , \quad \bq_{t_F}:=\bq(t_F,\brho),\quad\bv_{t_F}:=\bv(t_F,\brho).
\end{eqnarray}
The function $\tilde{\bg} : \Re^{1+2n+p} \to \Re^{\nc}$ is a vector of `trajectory cost functions', and $\tilde{w} : \Re^{1+2n+p} \to \Re^{\nc}$ is a vector of `terminal cost functions', and the system output $\psi \in \Re^{\nc}$ is a vector of $\nc$ `outputs', i.e., scalar cost functions. 
Both the trajectory and terminal cost functions can include accelerations via $\dbv$. Accelerations are not independent variables and can be resolved in terms of positions and velocities $\dbv:=\feom \left(t,\bq,\dbq,\brho\right)\in \mathds{R}^{n}$. The cost functions can also include arguments  
$\tilde{\bu}\bigl(\,t,\,\bq,\,\bv,\,\brho\,\bigr)=\bu\left(\,t,\,\bq,\,\bv,\,\dbv,\,\brho \right)$ that depend on the solution and on the acceleration. Our notation encompasses these cases by defining:
\begin{equation}
\label{eq:CostFunction-explained}
\begin{split}
\tilde{g}\bigl(\,t,\,\bq,\,\bv,\,\brho\,\bigr)&=g\bigl(\,t,\,\bq,\,\bv,\,\dbv,\,\brho, \, \bu (\,t,\,\bq,\,\bv, \,\dbv ,\brho) \bigr), \\
\tilde{w}\bigl(t_F,\bq_{t_F},\bv_{t_F},\brho\bigr)&={w}\left(t_F,\bq_{t_F},\bv_{t_F},\dbv_{t_F},\brho,\bu(t_F,\bq_{t_F},\bv_{t_F},\dbv_{t_F},\brho) \right).
\end{split}
\end{equation}
All functions are considered to be smooth.
\begin{definition}[Sensitivity analysis problem]
\label{def:sensitivity-problem}
The sensitivity analysis problem is to compute the derivatives of the model outputs \eqref{eq:CostFunction} with respect to model parameters:
\begin{equation}
\label{eqn:sensitivity-problem}
\frac{d\, \psi}{d\, \brho} := \displaystyle \begin{bmatrix} \displaystyle \frac{d\,\psi}{d\, \brho_1} \cdots \frac{d\,\psi}{d\, \brho_p} \end{bmatrix} \in \Re^{\nc \times p}.
\end{equation}
\end{definition}

\begin{definition}[The canonical ODE system] 
\label{def:canonical-ode}
To simplify the representation of the system we define the vector of `quadrature' variables $z \in \mathds{R}^{\nc}$ as follows:
\begin{equation} 
\label{eqn:quadrature-variable}
\begin{split}
\bz(t,\brho) &:=  \int_{t_0}^{t} { \tilde{g}\bigl(\tau,\bq,\bv,\brho\bigr) \,  d\tau} \quad \Leftrightarrow \quad \\
\dbz(t,\brho) &= \tilde{\bg}\bigl(t,\bq,\bv,\brho\bigr),  \quad t_0 \leq t \leq  t_F, \quad z(t_0,\brho)=0,
\end{split}
\end{equation}
 which leads the vector of cost function \eqref{eq:CostFunction}  at final time to become:
\begin{equation}
\label{eqn:quadrature-cost-function}
\psi =\bz(t_F) + \tilde{w}\bigl(t_F,\bq_{t_F},\bv_{t_F},\brho\bigr).
\end{equation}
Next, we add dummy evolution equations $\brho'=0$ for the time independent parameters.

Finally, we append the parameters and the quadrature variables to system state to obtain the following extended `canonical' state vector:
\[
\bx(t) := \begin{bmatrix} \bq(t)^T \quad \bv(t)^T \quad\ \brho(t)^T \quad \bz(t)^T \end{bmatrix}^T \in \mathds{R}^{(2  n + p+ \nc) \times 1} 
\]
together with the `canonical ODE system' that describes its evolution:
\begin{equation}
\label{eq:canonical-ode-system}
\dot{\bx} 
= \begin{bmatrix} \bv \\ \feom\bigl(t,\bq,\bv,\brho \bigr)  \\ \Zero_{p \times 1} \\\tilde{g}\fin \end{bmatrix} 
:= F(t,\bx) \in \mathds{R}^{2n+p+\nc}, \quad t_0 \le t \le t_F, \quad
\bx(t_0) := \begin{bmatrix} \bq_0(\brho) \\ \bv_0(\brho) \\ \brho \\ \Zero_{\nc \times 1} \end{bmatrix}.
\end{equation}
\end{definition}
%
\subsection{Direct sensitivity analysis for smooth ODE systems and extended cost function}
\label{sec:direct-sensitivity-continuous}
%
 Define the `position sensitivity' matrix $Q(t,\brho)$, the `velocity sensitivity' matrix  $V(t,\brho)$, the `quadrature sensitivity' matrix $Z(t,\brho)$, and an identity matrix $\Gamma$ as the formal sensitivity of the parameters, as:
\begin{subequations}
\label{eq:sensitivity-matrices}
\begin{eqnarray}
\label{eq:sensitivity-matrix-Q}
\bQ_i(t,\brho) &:= \displaystyle \frac{d\, \bq(t,\brho)}{d\, \brho_i} \in \mathds{R}^{n}, ~~ i=1,\dots,p; \quad
\bQ(t,\brho) &:= \begin{bmatrix} \bQ_1(t,\brho) \cdots \bQ_p(t,\brho) \end{bmatrix} \in \mathds{R}^{n \times p}, \\
\label{eq:sensitivity-matrix-V}
\bV_i(t,\brho) &:= \displaystyle\frac{d\, \bv(t,\brho)}{d\, \brho_i} \in \mathds{R}^{n}, ~~ i=1,\dots,p; \quad
\bV(t,\brho) &:= \begin{bmatrix} \bV_1(t,\brho) \cdots \bV_p(t,\brho) \end{bmatrix} \in \mathds{R}^{n \times p}, 
\\
\Gamma_i(t,\brho) &:= \displaystyle\frac{d\, \brho(t,\brho)}{d\, \brho_i} \in \mathds{R}^{p}, ~~ i=1,\dots,p; \quad
\Gamma(t,\brho) &:= \begin{bmatrix} \Gamma_1 \quad \quad \cdots \quad \quad  \Gamma_p \end{bmatrix}=\bI_{p \times p}, 
\\
Z_i(t,\brho) &:= \displaystyle\frac{\partial z(t,\brho)}{\partial \brho_i} \in \mathds{R}^{\nc}, ~~ i=1,\dots,p; \quad
Z(t,\brho) &:= \begin{bmatrix} Z_1(t,\brho) \cdots Z_p(t,\brho) \end{bmatrix} \in \mathds{R}^{\nc \times p}.
\end{eqnarray}
\end{subequations}
The direct sensitivity for ODE systems, referred to as the Tangent Linear Model (TLM), computes the sensitivity matrix
$\bX = \left[   
\bQ^{\rm T}, \,
\bV^{\rm T}, \,
\Gamma, \,
\bZ^{\rm T}
\right] ^{\rm T}  \in \mathds{R}^{(2n+p+\nc)\times p}$.
obtained by differentiating the canonical ODE system \eqref{eq:canonical-ode-system} with respect to the parameters:
\begin{eqnarray}
\label{eq:TLM-ODE-compact}
\dot{\bX}=\begin{bmatrix}
\dot{\bQ} \\
\dot{\bV} \\ 
\dot{\Gamma} \\
\dbZ
\end{bmatrix} 
=
\begin{bmatrix}
\bV \\
\feomq \bQ +
\feomv \bV +
\feomrho  \\
\Zero_{p \times p} \\
\tilde{\bg}_q \, \bQ +
\tilde{\bg}_v \, \bV +
\tilde{\bg}_{\brho} \\
\end{bmatrix}, \quad t_0 \le t \le t_F, \quad
\bX(t_0) := \begin{bmatrix} \displaystyle \frac{d \bq_0(\brho)}{d \brho} \\ \displaystyle \frac{d \bv_0(\brho)}{d \brho} \\ \bI_{p \times p} \\ \Zero_{\nc \times p} \end{bmatrix} \in \Re^{(2n+p+\nc) \times p}.
\end{eqnarray}
The direct sensitivity for ODE systems needs to be solved forward in time.
The expressions $\feomq$, $\feomv$, and $\feomrhoi$ denote the partial derivatives of $\feom$ with respect to the subscripted variables. The detailed calculation of these expressions and the remaining partial derivatives is explained in Appendix \ref{appendix:derivatives-calculation}. Once the sensitivities \eqref{eq:TLM-ODE-compact} have been calculated, the sensitivities of the cost functions \eqref{eqn:sensitivity-problem} with respect to parameters are computed as follows:
\begin{equation}
\label{eq:sensitivity cost function t_F}
\frac{d\,\psi}{d\, \brho} = \bZ(t_F) + 
\left[ \tilde{w}_\bq \cdot Q + \tilde{w}_\bv \cdot V  + \tilde{w}_{\brho} \right]_{t_F} \in \mathds{R}^{\nc \times p}. 
\end{equation}
We note that the TLM system \eqref{eq:TLM-ODE-compact} can be written in matrix form as follows:
\begin{equation}
\label{eq:TLM-ODE-matrix}
\begin{bmatrix}
 \dot{\bQ} \\\dot{\bV} \\ \dot{\Gamma} \\ \dbZ
\end{bmatrix} 
=
\begin{bmatrix}
\Zero_{n \times n} &	\textrm{I}_{n \times n} & \Zero_{n \times p} & \Zero_{n \times \nc} \\
	\feomq & \feomv &	\feomrho & \Zero_{n \times \nc} \\
	\Zero_{p \times n} & \Zero_{p \times n} & \Zero_{p \times p} & \Zero_{p \times \nc} \\
	 \tilde{\bg}_q & \tilde{\bg}_v & \tilde{\bg}_{\brho} & \Zero_{\nc \times \nc} \\
\end{bmatrix} \cdot
\begin{bmatrix}
{\bQ}\\{\bV} \\ \Gamma \\ \bZ
\end{bmatrix} , \quad t_0 \leq t \leq t_F.
\end{equation}
%

\subsection{Adjoint sensitivity analysis for smooth ODE systems and extended cost function}
In this section we provide the  system of equations that governs the adjoint sensitivity analysis for smooth ODE systems.

\begin{definition}[Adjoint sensitivity analysis]
Apply the chain rule differentiation to the total sensitivity of the cost function \eqref{eqn:sensitivity-problem}:
\begin{equation} 
\label{eqn:dpsi-chain-rule}
\frac{d\,\psi}{d\, \brho} =  \frac{d\,\psi}{d\, x(t,  \brho)} \cdot \frac{d\,x(t,  \brho)}{d\, \brho}=\lambda^{\rm T}(t,  \brho) \cdot X(t,  \brho), \quad 
\end{equation}
where $\lambda= (d\,\psi/d\, x)^T = \left[{\lambda^{\bQ}}^{\rm T},  {\lambda^{\bV}}^{\rm T},  {\lambda^{\Gamma}}^{\rm T}, {\lambda^{\bZ}}^{\rm T} \right]^{\rm T}$ is defined as:
\begin{subequations}
\label{eqn:adjoint-definition}
\begin{eqnarray}
\lambda^\bQ_j(t,  \brho)&:= \left(\frac{d\,\psi_j}{d\, \bq(t,  \brho)}\right)^T \in \mathds{R}^{ n \times 1}, ~~ j=1,\dots, \nc; \quad
\lambda^{\bQ}(t,\brho) &:= \begin{bmatrix} \lambda^{\bQ}_1(t,\brho)  \cdots \lambda^{\bQ}_{\nc}(t,\brho) \end{bmatrix} \in \mathds{R}^{n \times \nc}, \\
 \lambda^\bV_j(t,  \brho)&:= \left(\frac{d\,\psi_j}{d\, \bv(t,  \brho)}\right)^T \in \mathds{R}^{ n \times 1}, ~~ j=1,\dots, \nc; \quad
 \lambda^{\bV}(t,\brho) &:= \begin{bmatrix} \lambda^{\bV}_1(t,\brho)  \cdots \lambda^{\bV}_{\nc}(t,\brho) \end{bmatrix} \in \mathds{R}^{n \times \nc}, \\
 \lambda^\Gamma_j(t, \brho) &:= \quad \left(\frac{d\,\psi_j}{d\,\brho}\right)^T \in \mathds{R}^{ p \times 1}, ~~ j=1,\dots, \nc; \quad
 \lambda^{\Gamma}(t,\brho) &:= \begin{bmatrix} \lambda^{\Gamma}_1(t,\brho)  \cdots \lambda^{\Gamma}_{\nc}(t,\brho) \end{bmatrix} \in \mathds{R}^{p \times \nc} ,\\
  \lambda^\bZ_j(t,  \brho)&:= \left(\frac{d\,\psi_j}{d\, \bz(t,  \brho)}\right)^T \in \mathds{R}^{ \nc \times 1},  ~~ j=1,\dots, \nc; \quad
 \lambda^{\bZ}(t,\brho) &:= \begin{bmatrix} \lambda^{\bZ}_1  \quad \quad \cdots \quad \quad \lambda^{\bZ}_{\nc} \end{bmatrix} = \bI_{\nc \times \nc}.
\end{eqnarray}
\end{subequations}
Note that, from \eqref{eqn:quadrature-variable}--\eqref{eqn:quadrature-cost-function}
\[
\psi =\bz(t,\brho) + \int_{\tau}^{t_F} { \tilde{g}\bigl(\tau,\bq,\bv,\brho\bigr) \,  d\tau} + \tilde{w}\bigl(t_F,\bq_{t_F}\bigr),\]
which leads to the relation $d \psi/d \bz(t,\brho) = \bI_{\nc \times \nc}$ for any time $t$.

From \eqref{eqn:dpsi-chain-rule}  we have that for any time $t$:
\begin{equation} 
\label{eqn:dpsi-chain-rule-expanded}
\frac{d\,\psi}{d\, \brho} =  
{\lambda^{\bQ}(t,  \brho)}^{\rm T} \cdot  \bQ(t,  \brho) +  {\lambda^{\bV}(t,  \brho)}^{\rm T} \cdot  \bV(t,  \brho) +  {\lambda^{\Gamma}(t,  \brho)}^{\rm T} + {\lambda^{\bZ}(t,  \brho)}^{\rm T} \cdot  \bZ(t,  \brho).
\end{equation}
Evaluating \eqref{eqn:dpsi-chain-rule-expanded} at $t=t_F$ leads to the direct sensitivity approach:
\begin{equation}
\label{eqn:sensitivity-psi-direct}
\frac{d\,\psi}{d\, \brho} =  
{\lambda^{\bQ}(t_F,  \brho)}^{\rm T} \cdot  \bQ(t_F,  \brho) +  {\lambda^{\bV}(t_F,  \brho)}^{\rm T} \cdot  \bV(t_F,  \brho) +  {\lambda^{\Gamma}(t_F,  \brho)}^{\rm T} \cdot {\Gamma}(t_F,  \brho)  + {\lambda^{\bZ}(t_F,  \brho)}^{\rm T} \cdot  \bZ(t_F,  \brho).
\end{equation}
By comparing this equation with \eqref{eq:sensitivity cost function t_F} one obtains the values of the adjoint variables at the final time $t_F$:
\begin{equation} 
\label{eqn:adjoint-final-conditions}
{\lambda^{\bQ}(t_F,  \brho)} = \tilde{w}_\bq^{\rm T} \big|_{t_F}, \quad
{\lambda^{\bV}(t_F,  \brho)} = \tilde{w}_\bv^{\rm T} \big|_{t_F}, \quad
{\lambda^{\Gamma}(t_F,  \brho)} = \tilde{w}_{\brho}^{\rm T} \big|_{t_F}, \quad
{\lambda^{\bZ}(t_F,  \brho)} = \bI_{\nc \times \nc}.
\end{equation}
The equation \eqref{eqn:dpsi-chain-rule-expanded} evaluated at $t=t_F$ leads to the direct sensitivity approach: 
\begin{equation}
\label{eqn:sensitivity-psi-direct at tF}
\begin{split}
\frac{d\,\psi}{d\, \brho} &=  
\tilde{w}_\bq \big|_{t_F} \cdot  \bQ(t_F,  \brho) \, + \,  \tilde{w}_\bv \big|_{t_F} \cdot  \bV(t_F,  \brho) \, + \, \tilde{w}_{\brho}\big|_{t_F}  \cdot \bI_{p \times p} \, + \,
\bI_{\nc \times \nc}\cdot  \bZ(t_F,  \brho) \\
&=  \displaystyle 
\tilde{w}_\bq \big|_{t_F} \cdot  \bQ(t_F,  \brho) \, + \,  \tilde{w}_\bv \big|_{t_F} \cdot  \bV(t_F,  \brho) \, + \, \tilde{w}_{\brho}\big|_{t_F}   \, + \,
  \bZ(t_F,  \brho).
\end{split}
\end{equation}
Evaluating \eqref{eqn:dpsi-chain-rule-expanded} at $t=t_0$ leads to the adjoint sensitivity approach:
\begin{equation}
\label{eqn:sensitivity-psi-adjoint}
\begin{split}
\frac{d\,\psi}{d\, \brho} &=  
{\lambda^{\bQ}(t_0,  \brho)}^{\rm T} \cdot  \bQ(t_0,  \brho) +  
{\lambda^{\bV}(t_0,  \brho)}^{\rm T} \cdot  \bV(t_0,  \brho) \, + \,   
{\lambda^{\Gamma}(t_0,  \brho)}^{\rm T} \cdot \bI_{p \times p}  + 
{\lambda^{\bZ}(t_0,  \brho)}^{\rm T} \cdot  \bZ(t_0,  \brho) \\
&=  \displaystyle 
{\lambda^{\bQ}(t_0,  \brho)}^{\rm T} \cdot  \frac{d \bq_0(\brho)}{d \brho} \, + \,   {\lambda^{\bV}(t_0,  \brho)}^{\rm T} \cdot  \frac{d \bv_0(\brho)}{d \brho} \, + \,   {\lambda^{\Gamma}(t_0,  \brho)}^{\rm T} \cdot \bI_{p \times p} +
\bI_{\nc \times \nc}  \cdot  0_{n_v \times \nc} \\
&=  \displaystyle 
{\lambda^{\bQ}(t_0,  \brho)}^{\rm T} \cdot  \frac{d \bq_0(\brho)}{d \brho} \, + \,  {\lambda^{\bV}(t_0,  \brho)}^{\rm T} \cdot  \frac{d \bv_0(\brho)}{d \brho} \, + \, {\lambda^{\Gamma}(t_0,  \brho)}^{\rm T}  .
\end{split}
\end{equation}
Note that the adjoint variables are initialized at $t=t_F$. However, their values at $t=t_0$ are the ones needed for computing the desired sensitivities.
\end{definition}
\begin{definition}[The canonical adjoint sensitivity for ODE systems]
The evolution of adjoint variables for ODE systems is governed by the following continuous adjoint model: 
\begin{equation}
\label{eq:canonical-adjoint-ODE}
{\begin{bmatrix}
\dot{\lambda}^\bQ \\ 	\dot{\lambda}^\bV \\ \dot{\lambda}^{\Gamma} \\ \dot{\lambda}^{\bZ}
\end{bmatrix}}
= -
\begin{bmatrix}
\Zero_{n \times n} &	\feomq^\mathrm{T} & \Zero_{n \times p} & \tilde{\bg}_{\bq}^\mathrm{T} \\
\textrm{I}_{n \times n} &	\feomv^\mathrm{T} & \Zero_{n \times p} & \tilde{\bg}_v^\mathrm{T} \\
\Zero_{p \times n} &	\feomrho^\mathrm{T} & \Zero_{p \times p} & \tilde{\bg}_{\rho}^\mathrm{T} \\
\Zero_{\nc \times n}&	 \Zero_{\nc \times n} &  \Zero_{\nc \times p} & \Zero_{\nc \times \nc}
\end{bmatrix} 
\cdot 
{\begin{bmatrix}
\lambda^\bQ \\ 	\lambda^\bV \\ {\lambda}^{\Gamma} \\ {\lambda}^{\bZ}
\end{bmatrix}},	\quad t_F \geq t \geq t_0, \quad
\lambda(t_F, \rho) :=
\begin{bmatrix}
\tilde{w}_{q}^\mathrm{T}(t_F, \rho) \\
\tilde{w}_{v}^\mathrm{T}(t_F, \rho) \\
\tilde{w}_{{\rho}}^\mathrm{T}(t_F, \rho) \\ 
\textrm{I}_{\nc \times \nc}
\end{bmatrix} \in \mathds{R}^{ (2n+p+\nc) \times \nc }.
\end{equation}
The  adjoint sensitivities	\eqref{eq:canonical-adjoint-ODE} are solved backward in time.
\end{definition}
\subsection{Hybrid ODE system dynamics}
In this study, we consider hybrid ODE systems characterized by piecewise-in-time smooth dynamics described by \eqref{eq:EOM-ODE}, and that exhibit discontinuous dynamic behavior (jump or non-smoothness) in the generalized velocity state vector at a finite number of time moments (no zeno phenomena \cite{Pace2017}).  Each such moment is corresponds to an event  triggered by the event equation:
\begin{equation}
\label{eq:event_Function}
r \big( \qtimp \big) = 0,
\end{equation} 
where $\timp$  is the `time of event' and $r : \mathds{R}^{n} \to \mathds{R}$ is a smooth `event function'.  Note that grazing phenomena are not considered the in this study. The following quantities are used to characterize an event:
\begin{itemize}
	\item The value of a variable right before the event is denoted by
	$
	\boldsymbol{x}\beforeimp := \lim_{\varepsilon>0,~\varepsilon\to 0}\,\boldsymbol{x}(\timp-\varepsilon),
	$
	and its value right after the event by
	$
	\boldsymbol{x}\afterimp := \lim_{\varepsilon>0,~\varepsilon\to 0}\,\boldsymbol{x}(\timp+\varepsilon).
	$
	The limits exist since the evolution of the system is smooth in time both before and after the event.
	\item 
	The generalized position state variables remain the same after the event as before it, $\quad	\qplus = \qminus = \bq\atimp.$ This is a consequence of the event changing the energy of the system by a finite amount.
	\item Also due to the finite energy change during the event, the quadrature variable is continuous in time, $\quad \zplus= \zminus = z\atimp.$
	\item 
	An event that applies a finite energy impulse force to the system can abruptly change the generalized velocity state vector $\dbq$, from its value $\vminus$ right before the event to a new value $\vplus $ right after the event. The change in velocity is characterized by the `jump function': 
	\begin{equation}
	\label{eq:jump-in-velocity}
	\vplus=h \Big({\timp},\, \bq\atimp,\,\vminus,\,{\brho} \Big)  \quad \Leftrightarrow \quad  \dbq \afterimp=h \Big({\timp},\, \bq\atimp,\,\dbq  \beforeimp,\,{\brho} \Big) .
	\end{equation}
	\item 
	An event where the system undergoes a sudden change of the equation of motions \eqref{eq:EOM-ODE} at $\timp$ is characterized by the equations:
	\begin{equation}
	\label{eq:jump-in-eom}
	\ddbq \beforeimp = \feom^- \big(\timp,\qtimp,\vtimp, \brho\big) =:  \feom^-\atimp
	\quad \stackrel{\rm event}{\longrightarrow} \quad
	\ddbq \timpplus = \feom^+ \big(\timp,\qtimp,\vtimp, \brho\big)=:  \feom^+\atimp.
	\end{equation}
\end{itemize}

\begin{remark}[Multiple events]
In many cases, the change can be triggered by one of multiple events. Each individual event is described by the event function $r_\ell : \mathds{R}^n \to \mathds{R}$, $\ell=1,\dots,e$. The detection of the next event \eqref{eq:event_Function}, which can be one of the possible $e$ options, is described by $\Pi_{i=1}^e r_i \big(  \qtimp\big)= 0$, and if event $\ell$ takes place, then $r_\ell = 0$ and the corresponding jump in velocity \eqref{eq:jump-in-velocity} is $\vplus = h_\ell \Big(\qtimp,\,\vminus\Big)$, or the corresponding change in the equations of motion \eqref{eq:jump-in-eom} is $\ddbq \timpplus =:  \feom_\ell^+\atimp$.
\end{remark}

\subsection{Direct sensitivity analysis for hybrid ODE systems}
%
Let $\bQ\afterimp$ and $\bQ\beforeimp \in \mathds{R}^{n{\times}p}$  be the sensitivities of the generalized position state matrix after and before the event, respectively.
Let $\dqrhoplus,\dqrhominus  \in \mathds{R}^{n{\times}p}$ be the sensitivities of the generalized velocity state matrix after and before the event, respectively.
Let $\Zplus$ and $\Zminus$, with  $\bZ \in \mathds{R}^{p}$, be the sensitivities of the quadrature variable $\bz(t)$   after and before the event, respectively.
It is shown in \cite{Corner2017} that, at the time of the event, we have:
\begin{subequations}
\begin{itemize}
\label{eqn:forward-sensitivity-jumps}
\item
 The sensitivity of the time of event with respect to the system parameters is:
\begin{eqnarray}
\label{eq:sensitivity-time-of-event}
\frdtdrho=
- \, \dfrac{\displaystyle \frac{d r}{d\bq}\left(\qtimp\right)  \cdot \bQ\beforeimp}{\displaystyle \frac{d r}{d\bq}\left(\qtimp\right)  \cdot  \vminus }   
\in \mathds{R}^{1 \times p}.
\end{eqnarray}
where $dr/d\bq \in \mathds{R}^{1 \times n }$ is the  Jacobian of the event function.

\item
The jump equation of the sensitivities of the generalized position state vector is:
\begin{eqnarray}
\label{eq:jump-position-sensitivity}
\bQ\afterimp = \bQ\beforeimp - \bigg( \vplus - \vminus \bigg) \cdot \frdtdrho.
\end{eqnarray}

\item
The jump equation of the sensitivities of the generalized velocity state vector is:
\begin{eqnarray}
\label{eq:jump-in-velocity-sensitivity}
\dqrhoplus &=& h_\bq\beforeimp \cdot \qrhominus +h_\bv\beforeimp \cdot  \dqrhominus 
+\left( h_\bq\beforeimp  \cdot \vminus -\ddbq \timpplus + h_\bv\beforeimp  \cdot \ddbq \beforeimp  {+ h_t}\beforeimp \right) \cdot \frdtdrho {+ h_\brho\beforeimp},
\end{eqnarray}
where the Jacobians of the jump function are:
\begin{eqnarray}
h_t\beforeimp &:= \frac{\partial h}{\partial t}\big(\timp, \qtimp, \vminus ,\brho\big) \in \Re^{f \times 1}, \qquad h_\bq\beforeimp &:= \frac{\partial h}{\partial\bq}\big(\timp, \qtimp, \vminus ,\brho\big)\in \Re^{f \times n}, \nonumber \\
h_\bv\beforeimp &:= \frac{\partial h}{\partial\bv}\big(\timp,\qtimp, \vminus,\brho \big)\in \Re^{f \times f}, \qquad h_\brho\beforeimp &:= \frac{\partial h}{\partial\brho}\big(\timp,\qtimp, \vminus,\brho \big)\in \Re^{f \times p}.
\end{eqnarray}
\item 
The sensitivity of the cost function changes during the event is :
\begin{eqnarray}
\label{eq:jump-in-quadrature-sensitivity}
\Zplus= \Zminus - \Big( \gplus - \gminus \Big) \cdot \frdtdrho.
\end{eqnarray}
where 
\begin{equation}
\gplus := \gplusarg,
\quad
\gminus := \gminusarg,
\end{equation}
is the running cost function evaluated right after and right before the event, respectively. 
\end{itemize}
\end{subequations}
\begin{definition}[The generalized jump sensitivity matrix]
The direct sensitivity jump equations \eqref{eqn:forward-sensitivity-jumps} can be written compactly in matrix form as ${X}\afterimp= S \cdot{X}\beforeimp$, where $S$ is the generalized sensitivity jump matrix:
\begin{subequations}
\label{eqn:forward-sensitivity-jump-matrix}
\begin{eqnarray}
\label{eqn:Jump_direct_sensitivity}
\begin{bmatrix}
 \bQ\afterimp \\
\bV\afterimp \\
\Gamma\afterimp\\  
\bZ\afterimp 
\end{bmatrix} 
=
\underbrace{
\begin{bmatrix}
\left(\bQ\afterimp\right)_{\bQ\beforeimp}  & \Zero_{n \times n} &\Zero_{n \times p} & \Zero_{n \times \nc}\\
\left(\bV\afterimp\right)_{\bQ\beforeimp} &  h_\bv\beforeimp   &  h_{\brho}\beforeimp & \Zero_{n \times \nc}\\
\Zero_{p \times n} & \Zero_{p \times n}  & \bI_{p \times p} & \Zero_{p \times \nc}\\
\left(\bZ\afterimp\right)_{\bQ\beforeimp} & \Zero_{\nc \times n} & \Zero_{\nc \times p} &
\bI_{\nc \times \nc} \\
\end{bmatrix} 
}_{\mathsf{S}_\textnormal{eve}} \cdot \begin{bmatrix}
\bQ\beforeimp \\
\bV\beforeimp \\ 
\Gamma\beforeimp\\  
\bZ\beforeimp 
\end{bmatrix}. 
\end{eqnarray}
From \eqref{eq:sensitivity-time-of-event} we have that: 
\begin{eqnarray}
\left(\frdtdrho \right)_{\bQ\beforeimp}=
- \, \dfrac{\displaystyle \frac{d r}{d\bq}\left(\qtimp\right) }{\displaystyle \frac{d r}{d\bq}\left(\qtimp\right)  \cdot  \vminus }   
\in \mathds{R}^{1 \times n}.
\end{eqnarray}
The Jacobians $\left(\bQ\afterimp\right)_{\bQ\beforeimp}$ and $\left(\bV\afterimp\right)_{\bQ\beforeimp}$ are:
\begin{eqnarray}
\left(\bQ\afterimp\right)_{\bQ\beforeimp} &=&  \rm I - \bigg( \vplus - \vminus \bigg) \cdot \left(\frdtdrho \right)_{\bQ\beforeimp}  \\
\nonumber
&=&  \rm I + \bigg( \vplus - \vminus \bigg) \cdot\dfrac{\displaystyle \frac{d r}{d\bq}\left(\qtimp\right) }{\displaystyle \frac{d r}{d\bq}\left(\qtimp\right)  \cdot  \vminus }\in \mathds{R}^{n \times n}, 
\end{eqnarray}
and
\begin{eqnarray}
\left(\bV\afterimp\right)_{\bQ\beforeimp} &=& 
h_\bq\beforeimp   +\left( h_\bq\beforeimp  \cdot \vminus -\ddbq \timpplus + h_\bv\beforeimp  \cdot \ddbq \beforeimp  {+ h_t}\beforeimp \right) \cdot \left(\frdtdrho \right)_{\bQ\beforeimp}  \\
&=& h_\bq\beforeimp   - \left( h_\bq\beforeimp  \cdot \vminus -\ddbq \timpplus + h_\bv\beforeimp  \cdot \ddbq \beforeimp  {+ h_t}\beforeimp \right) \cdot \dfrac{\displaystyle \frac{d r}{d\bq}\left(\qtimp\right) }{\displaystyle \frac{d r}{d\bq}\left(\qtimp\right)  \cdot  \vminus } \in \mathds{R}^{n \times n},
\end{eqnarray}
respectively.
The Jacobian $\left(\bZ\afterimp\right)_{\bQ\beforeimp}$ is:
\begin{eqnarray}
\left(\bZ\afterimp\right)_{\bQ\beforeimp} = -  \Big( \gplus - \gminus \Big) \cdot \left(\frdtdrho \right)_{\bQ\beforeimp} 
= \Big( \gplus - \gminus \Big) \cdot\dfrac{\displaystyle \frac{d r}{d\bq}\left(\qtimp\right) }{\displaystyle \frac{d r}{d\bq}\left(\qtimp\right)  \cdot  \vminus } \in \mathds{R}^{\nc \times n}.
\end{eqnarray}
\end{subequations}
\end{definition}

\subsection{Adjoint sensitivity analysis for hybrid ODE unconstrained dynamical systems}
\begin{theorem}[Adjoint sensitivity jump matrix]
Let $\lambda^\bQ\beforeimp$ $\in \mathds{R}^{n \times \nc}$, $ \lambda^\bV\beforeimp \in \mathds{R}^{n \times \nc},  
{\lambda}^{\Gamma}\beforeimp  \in \mathds{R}^{p \times \nc} $ and $ {\lambda}^{\bZ}\beforeimp  \in \mathds{R}^{\nc \times \nc}$	
be the adjoint sensitivities before the time of event respectively, and 
$\lambda\beforeimp=\left[ {\lambda^\bQ\beforeimp}^{\rm T} \quad  {\lambda^\bV\beforeimp}^{\rm T} \quad {\lambda^{\Gamma}\beforeimp}^{\rm T} \quad  {\lambda^\bZ\beforeimp}^{\rm T} \right]^{\rm T}$.
Let $\lambda^\bQ\afterimp$ $\in \mathds{R}^{n \times \nc}$, $ \lambda^\bV\afterimp \in \mathds{R}^{n \times \nc},  
{\lambda}^{\Gamma}\afterimp  \in \mathds{R}^{p \times \nc} $ and $ {\lambda}^{\bZ}\afterimp  \in \mathds{R}^{\nc \times \nc}$
be the adjoint sensitivities after the time of event respectively, and
$\lambda\afterimp=\left[ {\lambda^\bQ\afterimp}^{\rm T} \quad  {\lambda^\bV\afterimp}^{\rm T} \quad {\lambda^{\Gamma}\afterimp}^{\rm T} \quad  {\lambda^\bZ\afterimp}^{\rm T} \right]^{\rm T}$.

The adjoint sensitivity jump equations at the time of an event are:
\begin{equation}
\label{eqn:adjoint-sensitivity-jump-matrix}
\lambda\beforeimp = \mathsf{S}_\textnormal{eve}^\textrm{T}  \cdot \lambda\afterimp \in \mathds{R}^{(2 \times n + p +\nc) \times \nc}
\end{equation}
where $\mathsf{S}_\textnormal{eve}^\textrm{T}$ is the transpose of the generalized sensitivity jump matrix \eqref{eqn:Jump_direct_sensitivity}.
\end{theorem}

\begin{proof}
	We start the proof from the following statement provided in \cite{stewart2010optimal} that mentions that the dot product of the sensitivity state matrix with the adjoint sensitive state matrix is constant at any time, 
	$
	\label{eq:dot product state sensitivity with the adjoint}
	{\lambda\afterimp}^\textrm{T} \cdot
	{\bX\afterimp}  =
	{\lambda\beforeimp}^\textrm{T} \cdot
	{\bX\beforeimp}.
	$
Using \eqref{eqn:Jump_direct_sensitivity}, the previous relationship is equivalent to
	$
	{\lambda\afterimp}^\textrm{T} \cdot
	\mathsf{S}_\textnormal{eve} \cdot
	{\bX\beforeimp} =
	{\lambda\beforeimp}^\textrm{T} \cdot
	{\bX\beforeimp}.
	$
Since this holds for any matrix $\bX\beforeimp$ it follows that
	$
	{\lambda\afterimp}^\textrm{T} \cdot \mathsf{S}_\textnormal{eve}  =
	{\lambda\beforeimp}^\textrm{T},
	$
which is equivalent to \eqref{eqn:adjoint-sensitivity-jump-matrix}.
\end{proof}
\begin{remark}
From \eqref{eqn:forward-sensitivity-jump-matrix} and \eqref{eqn:adjoint-sensitivity-jump-matrix} the adjoint sensitivity jump equations  for ODE systems without constraints are:
\begin{subequations}
\label{eqn:adjoint-sensitivity-jumps} 
\begin{align}
\lambda^\bQ\beforeimp
= &
\; \left(\bQ\afterimp\right)_{\bQ\beforeimp}^{\rm T}  \cdot  \lambda^\bQ\afterimp  +
\left(\bV\afterimp\right)_{\bQ\beforeimp}^{\rm T}  \cdot  	\lambda^\bV\afterimp +
\left(\bZ\afterimp\right)_{\bQ\beforeimp}^{\rm T}  \cdot  {\lambda}^{\bZ}\afterimp
\\
\lambda^\bV\beforeimp =& \;
(h_\bv\beforeimp)^{\rm T} \cdot \lambda^\bV\afterimp 
\\
{\lambda}^{\Gamma}\beforeimp =& \; (h_{\brho}\beforeimp)^{\rm T} \cdot  \lambda^\bV\afterimp  + 
{\lambda}^{\Gamma}\afterimp
\\
{\lambda}^{\bZ}\beforeimp = & \; {\lambda}^{\bZ}\afterimp
\end{align}
\end{subequations}
\end{remark}
\section{Sensitivity analysis for constrained multibody dynamical systems and extended cost functions}
\label{sec:sensitivity_constrained_systems}

\subsection{Representation of constrained multibody systems}
We consider constrained multibody systems that satisfy the following kinematic constraints: 
\begin{subequations}
\label{eq:ConstraintsEq}
\begin{eqnarray}
\label{eq:ConstraintsEq-position}
\bzero  &=& \bPhi, \\
\label{eq:ConstraintsEq-velocity}
\bzero &=& \dPhi =\dPhidq \, \dbq + \bPhi_t \quad \Rightarrow \quad  \dPhidq \bv = -\bPhi_t, \\
\label{eq:ConstraintsEq-acceleration}
\bzero &=& \ddPhi =\dPhidq \, \ddbq + \dPhidqq \, ( \dbq, \dbq ) + \dPhidtdq \, \dbq + \bPhi_{t, \, t} \quad \Rightarrow \quad  \dPhidq \, \dbv = -  ( \dPhidq  \, \bv)\, \bv  - \dPhidtdq \, \bv - \bPhi_{t, \, t} := \Faccel.
\end{eqnarray}
\end{subequations}
Here \eqref{eq:ConstraintsEq-position} is a holonomic position constraint equation $\bPhi(t,\bq,\brho)=\bzero$, where $\bPhi : \mathds{R}^{1+n+p} \to \mathds{R}^{m}$ is a smooth `position constraint' function. The velocity \eqref{eq:ConstraintsEq-velocity} and the acceleration \eqref{eq:ConstraintsEq-acceleration} kinematic constraints are found by differentiating the position constraint with respect to time.

\begin{remark} 
\label{adjoint of the algebraic variables}
Formalisms for constrained multibody systems may involve  Lagrangian coefficients ${\mu} : \Re^{1+2n+p} \to \Re^{m}$ that provide the necessary forces to satisfy the kinematic constraints \cite{Corner2017}. Our notation encompasses the case where the cost function penalizes the accelerations $\dbv$ and the joint forces via the Lagrangian coefficients $\mu$:
\begin{equation}
\label{eq:CostFunction-explained-ConstrainedSystem}
\begin{split}
\tilde{g}\bigl(\,t,\,\bq,\,\bv,\,\brho\,\bigr)&=g\bigl(\,t,\,\bq,\,\bv,\,\dbv(\,t,\,\bq,\,\bv, \, \brho),\,\brho, \, \mu (\,t,\,\bq,\,\bv, \, \brho) \bigr). 
\end{split}
\end{equation} 
It is shown in \ref{eq:CostFunction-explained} that the terminal cost function $\tilde{w}$ cannot directly depend on the acceleration $\dbv$ or on the Lagrange coefficients $\mu$, and therefore the derivatives are $\tilde{w}_{\dbv}=0$ and $\tilde{w}_{\mu}=0$. In a different notation, such result is also shown in \cite{Dopico2014}. Using equation \eqref{eq:sensitivity cost function t_F} we see that the final condition for the adjoint of the algebraic variables $\mu$ is zero,   $\lambda^{\Lambda}_{t_F}=  \left.\tilde{w}_{\mu} \right|_{t_F}=0$. 
\end{remark}
\subsection{Direct and adjoint sensitivity analysis for smooth systems in the penalty ODE formulation}
Define the extended mass matrix $\overline{\Mass}:\mathds{R} \times \mathds{R}^{n} \times \mathds{R}^{n} \times \mathds{R}^{p} \rightarrow \mathds{R}^{n \times n}$ and the extended right hand side function $\overline{\Force}:\mathds{R} \times \mathds{R}^{n} \times \mathds{R}^{n} \times \mathds{R}^{p} \rightarrow \mathds{R}^{n}$ as:
\begin{align}
\overline{\Mass}\left(t,\bq,\bv,\brho\right) &:= \Mass\left(t,\bq,\bv,\brho\right) +\dPhidq^{\rm T}\left(t,\bq,\bv,\brho\right)\cdot \alpha\cdot \dPhidq\left(t,\bq,\bv,\brho\right), 
\\
\overline{\Force}\left(t,\bq,\bv,\brho\right) &:= {\Force}\left(t,\bq,\bv,\brho\right)
-\dPhidq^{\rm T}\cdot \alpha\cdot \left(\dtdPhidq\, \bv +
\dPhi_{t}+2 \, \xi\, \omega\, \dPhi+\omega^2 {\bPhi}\right),
\end{align}
where $\alpha \in \Re^{m \times m}$ is the penalty factor of the ODE penalty formulation, $\xi\in \Re$ and $\omega\in \Re$ are the natural frequency and damping ratio coefficients of the formulation, respectively. The functions $\Phi$, $\dPhi$,  $\ddPhi$ $ :\mathds{R} \times \mathds{R}^{n} \times \mathds{R}^{n} \times \mathds{R}^{p} \rightarrow \mathds{R}^{m}$ are the position, velocity and acceleration kinematic constraints, respectively. The penalty formulation of a constrained rigid multibody system is written as a first order ODE:
\begin{equation}
\label{eq:penalty formulation}
\begin{cases}
\dbq &= \bv, \\
\dbv &= {\feom} \fin = \overline{\Mass}^{-1}\fin \cdot
\overline{\Force}\fin.
\end{cases}
\end{equation}
The Lagrange multipliers associated to the constraint forces are estimated as 
$
\label{eq:lambda}
\blambda^{*}=\alpha\, \left(\, \ddot{\bPhi}+
2 \, \xi\, \omega\, \dPhi+\omega^2 \,\bPhi \, \right)\,.
$
The sensitivities of the state variables of the system with respect to parameters evolve according to the tangent linear model derived in \cite{Sandu_2014_sensitivity_ODE_multibody,Sandu_2014_MBSVT,Zhu_2014_PhD,Sandu2015dynamic,Sandu_2017_vehicleOptimization}. Since the penalty formulation \eqref{eq:penalty formulation} evolves as an ODE, we can compute the direct sensitivities using \eqref{eq:TLM-ODE-matrix} with $\feomq=\overline{\Force}_\bq  - \overline{\Mass}_\bq \, \dbv, \quad 	\feomv=\overline{\Force}_\bv     , \quad \feomrho =
\overline{\Force}_\brho -\overline{\Mass}_{\brho} \, \dbv, \,$ as shown in Appendix \ref{appendix:derivatives-calculation}.
The derivatives 
$\overline{\Mass}_{\brho},$  $ \overline{\Mass}_\bq$, $\overline{\Force}_\bq $,  $\overline{\Force}_\bv$, and $ \overline{\Force}_\brho$  are given in \cite{Corner2017}.
Similarly, one can compute the adjoint sensitivities of the penalty formulation  \eqref{eq:penalty formulation} using \eqref{eq:canonical-adjoint-ODE}.

\subsection{Direct and adjoint sensitivity analysis for smooth systems in the index-1 differential-algebraic formulation}

\begin{definition}[Constrained multibody dynamics: the index-1 DAE formulation]
The index-1 formulation of the equations of motion is obtained by replacing the position constraint \eqref{eq:ConstraintsEq-position} with the acceleration constraint \eqref{eq:ConstraintsEq-acceleration}:
\begin{equation}
\label{eq:EOM-DAE-index1}
\begin{bmatrix}
\bI & \bzero & \bzero \\
\bzero & {\Mass}\left(t,\bq,\brho\right) & \dPhidq^{\rm T}\left(t,\bq,\brho\right) \\
\bzero & \dPhidq\left(t,\bq,\brho\right) & \bzero
\end{bmatrix}
\cdot
\begin{bmatrix}
\dbq \\ \dbv \\ \mu
\end{bmatrix}
=
\begin{bmatrix}
\bv \\
{\Force} \left(t,\bq,\bv,\brho\right)  \\
\Faccel\left(t,\bq,\bv,\brho\right)
\end{bmatrix},
\quad t_0 \leq t \leq  t_F,  \quad \bq(t_0)=\bq_0(\brho),\quad\bv(t_0)=\bv_0(\brho).
\end{equation}
The algebraic equation has the form $\fdaelb -\mu=0$.
\end{definition}

\begin{definition} [Tangent linear index-1 DAE]
\label{def:canonical-dae-sensitivity}
Sensitivities of solutions \eqref{eq:sensitivity-matrices} and multipliers: 
\begin{equation}
\label{eq:sensitivity-matrix-L}
\Lambda_i(t,\brho) := \displaystyle \frac{d\, \mu(t,\brho)}{d\, \brho_i} \in \mathds{R}^{m}, ~~ i=1,\dots,p; 
\end{equation}
of the system  \eqref{eq:EOM-DAE-index1} with respect to parameters evolve according to the tangent linear model derived in \cite{Sandu_2014_sensitivity_ODE_multibody,Sandu_2014_MBSVT,Zhu_2014_PhD,Sandu2015dynamic,Sandu_2017_vehicleOptimization}:
\begin{eqnarray}
\label{eq:canonical-DAE-sensitivity-long}
\begin{bmatrix}
\dot{\bQ} \\
\dbV \\ 
\dot{\Gamma} \\
\Lambda \\ 
\dbZ
\end{bmatrix}=
\begin{bmatrix}
\bV  \\
\fdaedvq \bQ +\fdaedvv \bV +\fdaedvrho \\
\Zero_{p \times p} \\
\fdaelbq \bQ +\fdaelbv \bV +\fdaelbrho \\
\big(g_\bq +  g_{\dbv} \, \fdaedvq +g_{\mu} \, \fdaelbq \big)\cdot \bQ  
+  \big(g_\bv +  g_{\dbv} \, \fdaedvv +g_{\mu} \, \fdaelbv \big)\cdot \bV
+  \big(g_{\brho}+ g_{\dbv}\,\fdaedvrho +g_{\mu}\,\fdaelbrho \big)   
\end{bmatrix}.
\end{eqnarray}
It is shown in Appendix \ref{appendix:derivatives-calculation} that equation \eqref{eq:canonical-DAE-sensitivity-long} can be written in matrix form as follows:
\begin{eqnarray}
\label{eq:canonical-DAE-sensitivity}
\begin{bmatrix}
\dot{\bQ} \\
\dbV \\ 
\dot{\Gamma} \\
\Lambda \\ 
\dbZ
\end{bmatrix}
=
\begin{bmatrix}
\Zero_{n \times n} & \bI_{n \times n} & \Zero_{n \times p} & \Zero_{n \times m}  & \Zero_{n \times \nc} \\
\fdaedvq & \fdaedvv & \fdaedvrho & \Zero_{n \times m} & \Zero_{n \times \nc} \\
\Zero_{p \times n} & \Zero_{p \times n} & \Zero_{p \times p} & \Zero_{p \times m}  
& \Zero_{p \times \nc} \\
\fdaelbq  & \fdaelbv  & \fdaelbrho  & \Zero_{m \times m}  & \Zero_{m \times \nc}  \\
\tilde{\bg}_{\bq}  & \tilde{\bg}_{\bv}  & \tilde{\bg}_{\brho}  & \Zero_{\nc \times m}  
& \Zero_{\nc \times \nc}  \\
\end{bmatrix} \cdot
\begin{bmatrix}
\bQ \\
\bV \\ 
\Gamma \\
\Lambda \\ 
\bZ
\end{bmatrix},
\end{eqnarray}
with initial conditions given by Eq. \eqref{eq:TLM-ODE-compact}. Using Appendix \ref{appendix:derivatives-calculation}, the derivatives of the DAE function are:
\[
\fdaeq=  
\begin{bmatrix}
{\Mass} & \dPhidq^{\rm T} \\
 \dPhidq & \bzero
\end{bmatrix}^{-1}
\begin{bmatrix}
\Force_\bq -\Mass_{\bq} \, \dbv - \dPhidqq ^{\rm T}\, \mu 
\\
\Faccel_{\bq}- \dPhidqq \dot{\bv}   
\end{bmatrix}
, \;
\fdaev= \begin{bmatrix}
{\Mass} & \dPhidq^{\rm T} \\
 \dPhidq & \bzero
\end{bmatrix}^{-1}
\begin{bmatrix}
\Force_\bv \\
\Faccel_\bv 
\end{bmatrix}
,\;
\fdaerho=\begin{bmatrix}
{\Mass} & \dPhidq^{\rm T} \\
 \dPhidq & \bzero
\end{bmatrix}^{-1}
\begin{bmatrix}
\Force_{\brho}-\Mass_{\brho} \, \dbv 
- \dPhidqdrho^{T}\, \mu\\
\Faccel_{\brho} - \dPhidqdrho \, \dbv
\end{bmatrix}.
\]
\end{definition}
\begin{definition} [Continuous adjoint index-1 DAE system]
\label{def:adjoint-index1}
The continuous adjoint differential equation corresponding to the index-1 DAE tangent linear model \eqref{eq:canonical-DAE-sensitivity} is:
\begin{eqnarray}
\label{eq:Adjoint-DAE-sensitivity}
\begin{bmatrix}
\dot{{\lambda}^{\bQ}} \\
\dot{{\lambda}^{\bV}}\\ 
\dot{{\lambda}^{\Gamma}} \\
{{\lambda}^{\Lambda}} \\ 
\dot{{\lambda}^{\bZ}}
\end{bmatrix}
=
- \begin{bmatrix}
\Zero_{n \times n} & \fdaedvq^{\rm T} & \Zero_{n \times p} & \fdaelbq^{\rm T}  
& \tilde{\bg}_{\bq}^{\rm T} \\
  \bI_{n \times n} & \fdaedvv^{\rm T}  & \Zero_{n \times p}  & \fdaelbv^{\rm T}  
  & \tilde{\bg}_{\bv}^{\rm T} \\
   \Zero_{p \times n}   & \fdaedvrho^{\rm T} & \Zero_{p \times p}  & \fdaelbrho^{\rm T}   & \tilde{\bg}_{\brho}^{\rm T} \\
   \Zero_{m \times n}   & \Zero_{m \times n}  & \Zero_{m \times p}   & \Zero_{m \times m}
      & \Zero_{m \times \nc}  \\
   \Zero_{\nc \times n}  & \Zero_{\nc \times n}  & \Zero_{\nc \times p}    & \Zero_{\nc \times m} 
      & \Zero_{\nc \times \nc}   \\
\end{bmatrix} \cdot
\begin{bmatrix}
{{\lambda}^{\bQ}} \\
{{\lambda}^{\bV}}\\ 
{{\lambda}^{\Gamma}} \\
{{\lambda}^{\Lambda}} \\ 
{{\lambda}^{\bZ}}
\end{bmatrix},	\quad t_F \geq t \geq t_0, \quad
\lambda(t_F, \rho) :=
\begin{bmatrix}
\tilde{w}_{q}^\mathrm{T}(t_F, \rho) \\
\tilde{w}_{v}^\mathrm{T}(t_F, \rho) \\
\tilde{w}_{{\rho}}^\mathrm{T}(t_F, \rho) \\ 
0_{m \times \nc} \\
\textrm{I}_{\nc \times \nc}
\end{bmatrix} \in \mathds{R}^{ (2n+p+m+\nc) \times \nc }.
\end{eqnarray}
Noting  from Remark \ref{adjoint of the algebraic variables} that the algebraic equation in \eqref{eq:Adjoint-DAE-sensitivity} reads:
\[
{\lambda}^{\Lambda}(t) = 0,\quad t_F \geq t \geq t_0,
\]
the index-1 adjoint DAE  \eqref{eq:Adjoint-DAE-sensitivity} can be reduced to the following adjoint ODE:
\begin{eqnarray}
\label{eq:Adjoint-DAE-sensitivity1}
\begin{bmatrix}
\dot{{\lambda}^{\bQ}} \\
\dot{{\lambda}^{\bV}}\\ 
\dot{{\lambda}^{\Gamma}} \\
\dot{{\lambda}^{\bZ}}
\end{bmatrix}
=
- \begin{bmatrix}
\Zero_{n \times n} & \fdaedvq^{\rm T} & \Zero_{n \times p}   
& \tilde{\bg}_{\bq}^{\rm T} \\
\bI_{n \times n} & \fdaedvv^{\rm T}  & \Zero_{n \times p}  
& \tilde{\bg}_{\bv}^{\rm T} \\
\Zero_{p \times n}   & \fdaedvrho^{\rm T} & \Zero_{p \times p}   & \tilde{\bg}_{\brho}^{\rm T} \\
\Zero_{\nc \times n}  & \Zero_{\nc \times n}  & \Zero_{\nc \times p}    
& \Zero_{\nc \times \nc}   \\
\end{bmatrix} \cdot
\begin{bmatrix}
{{\lambda}^{\bQ}} \\
{{\lambda}^{\bV}}\\ 
{{\lambda}^{\Gamma}} \\
{{\lambda}^{\bZ}}
\end{bmatrix},	\quad t_F \geq t \geq t_0, \quad
\lambda(t_F, \rho) :=
\begin{bmatrix}
\tilde{w}_{q}^\mathrm{T}(t_F, \rho) \\
\tilde{w}_{v}^\mathrm{T}(t_F, \rho) \\
\tilde{w}_{{\rho}}^\mathrm{T}(t_F, \rho) \\ 
\textrm{I}_{\nc \times \nc}
\end{bmatrix} \in \mathds{R}^{ (2n+p+\nc) \times \nc }.
\end{eqnarray}

\end{definition}

\subsection{Direct sensitivity analysis for hybrid constrained dynamical systems}
\label{sec:multibody-constrained-direct}

We now discuss constrained dynamical systems when the dynamics is piecewise smooth in time. Performing a sensitivity analysis for a constrained rigid hybrid multibody dynamic system requires finding the jump conditions at the time of event. These jump equations are explained in our previous work \cite{Corner2017}. We summarize below the jump equations at the time of event:
\begin{itemize}
\item The generalized position state variables remain the same , i.e., $\bq\afterimp = \bq\beforeimp = \qtimp$ and need to satisfy both constraint functions
$
\bPhi^-\beforeimp:=\bPhi^-\left(\timp,\qtimp, \brho \right) = 0,$ and $ \,
\bPhi^+\timpplus:=\bPhi^+\left(\timp,\qtimp, \brho \right) = 0.
$

\item The velocity state variables jump from their values right  before the event to right after the event according to the jump equation:
\begin{equation}
\label{eq:jump-in-velocity-constrained-dof}
\bv_\textnormal{dof+}\afterimp =h \Big({\timp} , \qtimp, \bv_\textnormal{dof-}\beforeimp, {\brho} \Big), \qquad
h : \mathds{R}^{1+n+f^-+p} \to  \mathds{R}^{f^+}.
\end{equation}
The jump function \eqref{eq:jump-in-velocity-constrained-dof} is assumed to be smooth and defined in terms of the velocity degrees of freedom (the independent components). 

\item  The jumps in velocity  cannot be arbitrary for the dependent components. They are dependent of the degree of freedom and are obtained from solving the velocity constraints leading to:
\begin{equation}
\label{eq:jump-in-velocity-constrained-dep}
\begin{split}
 \bv_{\rm dep+}\afterimp &= -\left(\bPhi^+_{\bq_{\rm dep+}}\afterimp\right)^{-1}\cdot \left( \bPhi^+_{\bq_{\rm dof+}}\afterimp\, \bv_{\rm dof+}\afterimp  +  \bPhi^+_t\afterimp \right) \\
 &= \Rez^+\afterimp\,  \bv_{\rm dof+}\afterimp -\left(\bPhi^+_{\bq_{\rm dep+}}\afterimp\right)^{-1}\cdot   \bPhi^+_t\afterimp.
 \end{split}
\end{equation}
Where $\Rez^\pm$ corresponds to the null space of the constraints if the constraints are scleronomic (non explicitly time dependent).
\end{itemize}
There are two types of velocity jumps that our formalism covers \eqref{eq:jump-in-velocity-constrained-dof}--\eqref{eq:jump-in-velocity-constrained-dep}:
\begin{itemize}
\item The case where the event consists of an elastic contact/collision/impact on the DOF components of the velocity state. The impulsive (external) contact forces act to change the DOF components without changing the set of constraint equations,  $\bPhi^+\equiv\bPhi^-$. 
	
\item The case where the event consists solely of an inelastic collisions and a change of constraints $\bPhi^+\ne\bPhi^-$, without any external force modifying the independent velocities.
The impulsive (internal) constraints forces at the time of event  are solved by using a popular approach in robotics \cite{Kolathaya2016}:
\begin{subequations}
\label{eq:DAE-impulse}
\begin{equation}
\label{eq:DAE-impulse-1}
\begin{bmatrix}
{\Mass\atimp} & (\dPhidq^+)^{\rm T}\atimp \\
\dPhidq^+\atimp & \bzero
\end{bmatrix}
\cdot
\begin{bmatrix}
\vplus \\ \delta\mu
\end{bmatrix}
=
\begin{bmatrix}
 \Mass\atimp \cdot \vminus \\
-\bPhi_t^+\atimp 
\end{bmatrix},
\end{equation}
or, equivalently,
\begin{equation}
\label{eq:DAE-impulse-2}
\begin{bmatrix}
\vplus \\ \delta\mu
\end{bmatrix}
=
\begin{bmatrix}
{\Mass\atimp} & (\dPhidq^+)^{\rm T}\atimp \\
\dPhidq^+\atimp & \bzero
\end{bmatrix}^{-1}
\cdot
\begin{bmatrix}
 \Mass\atimp \cdot \vminus \\
-\bPhi_t^+\atimp 
\end{bmatrix}=
\begin{bmatrix}
\fdaeimpv {\left(\,\timp,\,\qtimp ,\,\vminus,\,\brho \, \right)} \\
\fdaeimpl {\left(\,\timp,\,\qtimp ,\,\vminus,\,\brho \, \right)}
\end{bmatrix}.
\end{equation}
\end{subequations}
The second equation  \eqref{eq:DAE-impulse-2} imposes the velocity constraint on both independent and dependent coordinates, which is covered by our formalism as:
\begin{equation}
\bv_{\rm dof+}\afterimp = \Permutation_{\rm dof+}\,\fdaeimpv {\left(\,\timp,\,\qtimp ,\,\vminus,\,\brho \, \right)}
=: h\left(\,\timp,\,\qtimp ,\,\bv_{\rm dof-}\beforeimp,\,\brho \, \right),
\end{equation}
where $\Permutation = \begin{bmatrix} \Permutation_{\rm dep} \\ \Permutation_{\rm dof} \end{bmatrix}$ is a permutation matrix that partitions the state variables into dependent and independent variables.
\end{itemize}

Finally, the jump conditions at the time of event in the sensitivity state matrix are: 
\begin{itemize}

\item  The independent components of the sensitivity of the generalized positions right after the event: 
\begin{subequations}
\label{eq:jump-position-sensitivity-constrained}
\begin{equation}
\label{eq:jump-position-sensitivity-constrained-dof}
\bQ_\textnormal{dof+}\afterimp = \bQ_\textnormal{dof+}\beforeimp  - \bigg( \bv_\textnormal{dof+}\afterimp - \bv_\textnormal{dof+}\beforeimp \bigg) \cdot \frdtdrho.
\end{equation}
which are equivalent to:
\begin{equation}
\Permutation^+_\textnormal{dof+} \cdot \bigg(\bQ\afterimp - \bQ\beforeimp  \bigg)
= - \Permutation^+_\textnormal{dof+} \cdot \bigg( \bv\afterimp - \vminus\bigg) \cdot \frdtdrho.
\end{equation}

\item  The dependent components of the sensitivity of the generalized positions right after the event:
\begin{equation}
\label{eq:jump-position-sensitivity-constrained-dep}
\bQ_\textnormal{dep+}\afterimp =  \Rez^+\afterimp \cdot \bQ_\textnormal{dof+}\afterimp - \left. \left(\bPhi^+_{\bq_\textnormal{dep+}}\afterimp\right)^{-1}\, \bPhi^+_\brho \right\afterimp.
\end{equation}
\end{subequations}

\item The independent coordinates of the velocity sensitivities right after the event,
\begin{align}
\label{eq:jump-in-velocity-sensitivity-constrained-dof}
\bV_\textnormal{dof+}\afterimp &= h_{\bq}\beforeimp \cdot \qrhominus 
+h_{\bv_\textnormal{dof-}}\beforeimp \cdot  \bV_\textnormal{\rm dof-}\beforeimp \\ & \nonumber
+\Bigl( h_{\bq}\beforeimp\cdot \vminus-\ddbq_\textnormal{dof+}\afterimp +h_{\bv_\textnormal{dof-}}\beforeimp\cdot \ddbq_\textnormal{dof-}\beforeimp {+ h_t\beforeimp} \Bigr) \cdot \frdtdrho
{+ h_\brho\beforeimp},
\end{align}
where the Jacobians of the jump function are:
\[
\begin{split}
h_{\bq}\beforeimp &:= \frac{\partial h}{\partial\bq}\big( \bq\atimp,\bv_\textnormal{dof-}\beforeimp, \brho \big) \in \Re^{f^+ \times n},
\qquad
h_{\bv_\textnormal{dof-}}\beforeimp := \frac{\partial h}{\partial\bv_\textnormal{dof-}}\big( \bq\atimp,\bv_\textnormal{dof-}\beforeimp, \brho \big) \in \Re^{f^+ \times f^-}.
\\
h_t\beforeimp &:= \frac{\partial h}{\partial t}\big( \bq\atimp,\bv_\textnormal{dof-}\beforeimp, \brho \big) \in \Re^{f}, \qquad \;\;\;\;
h_\brho\beforeimp := \frac{\partial h}{\partial\brho}\big( \bq\atimp,\bv_\textnormal{dof-}\beforeimp, \brho \big)\in \Re^{f \times p}.
\end{split}
\]
\item The dependent components  of the velocity sensitivities right after the event,
\begin{equation}
\label{eq:jump-in-velocity-sensitivity-constrained-dep}
\bV_\textnormal{dep+}\afterimp =
\left. - \big( \bPhi^+_{\bq_\textnormal{dep+}} \afterimp \big)^{-1}
 \left(
\bPhi^+_{\bq_\textnormal{dof+}} \cdot \bV_\textnormal{dof+} 
 + \big(\dPhidqqplus \,\bv +  \dPhidtdqplus \big) \cdot \bQ 
 +  \dPhidqdrhoplus \, \bv  + \dPhidtdrhoplus
 \right)\right\afterimp.
\end{equation}
\end{itemize}

\begin{definition}[The generalized sensitivity jump matrix for elastic impact]
The jump equations \eqref{eq:jump-in-velocity-constrained-dof}--\eqref{eq:jump-in-velocity-sensitivity-constrained-dep} for constrained systems can be written compactly in matrix form as a jump of the state sensitivity matrix ${X}$ at the time of the event,  ${X}\afterimp= \mathsf{S}_\textnormal{eve} \cdot{X}\beforeimp$, where $\mathsf{S}_\textnormal{eve}$ represents the generalized jump sensitivity matrix:
\begin{eqnarray}
\label{eqn:Jump_direct_sensitivity_constraints}
\begin{bmatrix}
\bQ_\textnormal{dep+}\afterimp\\
\bQ_\textnormal{dof+}\afterimp \\  
\bV_\textnormal{dep+}\afterimp\\
\bV_\textnormal{dof+}\afterimp \\
\Gamma\afterimp\\ 
\bZ\afterimp
\end{bmatrix}
=
\underbrace{
\begin{bmatrix}
\left(\bQ_\textnormal{dep+}\afterimp \right)_{\bQ\beforeimp }   & \Zero_{(n-f)\times (n-f)}  & \Zero_{(n-f)\times f}  &  D& \Zero_{(n-f)\times \nc} \\
\left(\bQ_\textnormal{dof+}\afterimp \right)_{\bQ\beforeimp}   & \Zero_{f\times (n-f)}  & \Zero_{f\times f}  &  \Zero_{f\times p}   & \Zero_{f\times \nc} \\
\left(\bV_\textnormal{dep+}\afterimp \right)_{\bQ\beforeimp }  & \Zero_{(n-f)\times (n-f)}  &  \left(\bV_\textnormal{dep+}\afterimp \right)_{\bV_\textnormal{dof+}\beforeimp }    &  K  & \Zero_{(n-f)\times \nc} \\
\left(\bV_\textnormal{dof+}\afterimp \right)_{\bQ\beforeimp}  &  \Zero_{f\times (n-f)}  & h_\bv\beforeimp   &  h_{\brho}\beforeimp  & \Zero_{f\times \nc} \\
\Zero_{p\times n}  & \Zero_{p\times (n-f)}   & \Zero_{p\times f}  &  \bI_{p\times p}   & \Zero_{p\times \nc} \\
\left(\bZ\afterimp\right)_{\bQ\beforeimp} & \Zero_{\nc \times (n-f)}  & \Zero_{\nc \times f} & 
\Zero_{\nc \times p}  & \bI_{\nc \times \nc}  
\end{bmatrix} 
}_{\mathsf{S}_\textnormal{eve}}
\cdot
\begin{bmatrix}
\bQ_\textnormal{dep+}\beforeimp\\
\bQ_\textnormal{dof+}\beforeimp \\
\bV_\textnormal{dep+}\beforeimp\\
\bV_\textnormal{dof+}\beforeimp \\
\Gamma\beforeimp \\  
\bZ\beforeimp \\
\end{bmatrix}.
\end{eqnarray}
The Jacobians of the jump equations with respect to the sensitivity state before the time of event are:
\begin{eqnarray}
\left(\bQ_\textnormal{dof+}\afterimp\right)_{\bQ\beforeimp} =  \Permutation^+_\textnormal{dof+} 
\left(\rm \bI_{n \times n} -  \bigg( \vplus - \vminus \bigg)  \cdot 
\left(\frdtdrho \right)_{\bQ\beforeimp} \right) (\Permutation^-)^{\rm T} \in \mathds{R}^{f \times n},
\end{eqnarray}
with
\begin{eqnarray}
\left(\frdtdrho \right)_{\bQ\beforeimp}=
- \, \dfrac{\displaystyle \frac{d r}{d\bq}\left(\qtimp\right) }{\displaystyle \frac{d r}{d\bq}\left(\qtimp\right)  \cdot  \vminus }   
\in \mathds{R}^{1 \times n}.
\end{eqnarray}
It follows that:
\begin{equation}
\left(\bQ_\textnormal{dep+}\afterimp \right)_{\bQ\beforeimp }  =  \Rez^+\afterimp \cdot  \left(\bQ_\textnormal{dof+}\afterimp\right)_{\bQ\beforeimp} \in \mathds{R}^{(n-f) \times n},
\end{equation}
and
\begin{equation}
\begin{split}
\left( \bV_\textnormal{dof+}\afterimp \right)_{\bQ\beforeimp }  = 
\left(h_{\bq}\beforeimp
+\Bigl( h_{\bq}\beforeimp\cdot \vminus-\ddbq_\textnormal{dof+}\afterimp +h_{\bv_\textnormal{dof-}}\beforeimp\cdot \ddbq_\textnormal{dof-}\beforeimp {+ h_t\beforeimp} \Bigr) \cdot \left(\frdtdrho \right)_{\bQ\beforeimp} \right)\cdot (\Permutation^-)^{\rm T} \in \mathds{R}^{f \times n}
\end{split}.
\end{equation}
Rewriting  \eqref{eq:jump-in-velocity-sensitivity-constrained-dep} as:
\begin{equation}
\left(\bV_\textnormal{dep+}\afterimp \right)=
\left.  \left( \Rez^+ \cdot \bV_\textnormal{dof+} 
+  {\overline \Rez^+} \cdot \bQ 
+  C \right)
\right\afterimp\in \mathds{R}^{(n-f)\times p}
\end{equation}
or, equivalently, as: 
\begin{equation}
\begin{split}
\left(\bV_\textnormal{dep+}\afterimp \right) &=
\left.  \left( \Rez^+\cdot \bV_\textnormal{dof+} 
+  {\overline \Rez^+ } \cdot 
(\Permutation^-)^{\rm T} \cdot 
\begin{bmatrix}
\bQ_\textnormal{dep+}\beforeimp\\
\bQ_\textnormal{dof+}\beforeimp
\end{bmatrix}
+  C \right)
\right\afterimp, \\
{\overline \Rez^+\afterimp } &= \left. - \big( \bPhi^+_{\bq_\textnormal{dep+}} \big)^{-1}  \left( \dPhidqqplus \,\bv +  \dPhidtdqplus  \right)\right\afterimp, \\
{ C} &=\left. - \big( \bPhi^+_{\bq_\textnormal{dep+}}  \big)^{-1}
\left(  \dPhidqdrhoplus \, \bv  + \dPhidtdrhoplus
\right)\right\afterimp \in \mathds{R}^{(n-f)\times p},
\end{split}
\end{equation}
we find the following expressions for the Jacobians:
\begin{equation}
\begin{split}
\left(\bV_\textnormal{dep+}\afterimp \right)_{\bQ\beforeimp } &=
 \Rez^+\afterimp \cdot \left(\bV_\textnormal{dof+}\afterimp \right)_{\bQ\beforeimp}  
+  {\overline \Rez^+\afterimp } \cdot 
(\Permutation^-)^{\rm T} 
\quad \mathds{R}^{(n-f)\times n}, \\
\left(\bV_\textnormal{dep+}\afterimp \right)_{\bV_\textnormal{dof+}\beforeimp } &=
 \Rez^+\afterimp \cdot  h_\bv\beforeimp.
\end{split}
\end{equation}

The expressions for $D$ and $K$  in \eqref{eqn:Jump_direct_sensitivity_constraints} are:
\begin{equation}
\begin{split}
D &=  - \left. \left(\bPhi^+_{\bq_\textnormal{dep+}}\right)^{-1}\, \bPhi^+_\brho \right\afterimp
\in \mathds{R}^{(n-f)\times p}, \\
K &=\left. \left( C + \Rez^+ \cdot h_\brho\beforeimp+ \overline \Rez^+  \cdot
(\Permutation^-)^{\rm T}  \cdot
\begin{bmatrix}  
D \\  
\Zero_{f \times p} \\ 
\end{bmatrix} \;  \right) \;
\right\afterimp \in \mathds{R}^{(n-f)\times p}.
\end{split}
\end{equation}

\end{definition}

\begin{definition}[The generalized sensitivity jump matrix for inelastic impact with a sudden change of constraints]

Consider the event consisting of an inelastic collision and a sudden change of constraints  \eqref{eq:DAE-impulse}. 
The jump in the velocity sensitivity for constrained systems due to impulsive forces, presented in \cite{Corner2017}, is determined as follows: 
\begin{equation}
\begin{split}
\begin{bmatrix}
\dqrhoplus \\ \delta\Lambda
\end{bmatrix}
=&
-
\begin{bmatrix}
{\Mass} & \dPhidqplus^{\rm T} \\
\dPhidqplus & \bzero
\end{bmatrix}^{-1}
\begin{bmatrix}
{\Mass}_{\bq} \cdot (\vplus - \vminus) + \dPhidqqplus^{\rm T} \cdot \delta\lambda \\
\dPhidqqplus \cdot \vplus
\end{bmatrix}\cdot \qrhoplus +
\begin{bmatrix}
{\Mass} & \dPhidqplus^{\rm T} \\
\dPhidqplus & \bzero
\end{bmatrix}^{-1}
\begin{bmatrix}
\Mass \\ 
\bzero
\end{bmatrix} \cdot \dqrhominus
\\& -
\begin{bmatrix}
{\Mass} & \dPhidqplus^{\rm T} \\
\dPhidqplus & \bzero
\end{bmatrix}^{-1}
\begin{bmatrix}
{\Mass}_{\brho} \cdot \vplus +  \dPhidqdrhoplus^{\rm T} \cdot \delta\lambda \\
\dPhidqdrhoplus \cdot \vplus  + \dPhidtdrhoplus \cdot \vminus
\end{bmatrix}
-
\begin{bmatrix}
{\Mass} & \dPhidqplus^{\rm T} \\
\dPhidqplus & \bzero
\end{bmatrix}^{-1}
\begin{bmatrix}
\bzero \\
\dPhidtdqplus \cdot \vminus
+ \dPhidtdvplus \cdot \vminus  
\end{bmatrix},
\end{split}
\end{equation}
which simplifies to:
\begin{equation}
\begin{bmatrix}
\dqrhoplus \\ \delta\Lambda
\end{bmatrix}
=
\fun_\bq^{\scalebox{0.6}{\rm DAE-imp}} \cdot \qrhoplus 
+\fun_{\vminus}^{\scalebox{0.6}{\rm DAE-imp}} \cdot \dqrhominus
+ \fun_{\brho}^{\scalebox{0.6}{\rm DAE-imp}} 
+ \fun_{t}^{\scalebox{0.6}{\rm DAE-imp}} .
\end{equation}
Thus, the jump the velocity state variables  at the time of event is 
\begin{equation}
\bV\afterimp  =\fdaeimpvbq \afterimp \; \bQ\afterimp +\fdaeimpvbv\afterimp \; \bV\afterimp +\fdaeimpvbrho\afterimp, 
\end{equation}
and the jump in the sensitivity of the Lagrange multipliers from $\Lambda\beforeimp \to \Lambda\afterimp$  is:
\begin{equation}
\Lambda\afterimp  =\Lambda\beforeimp  +\fdaeimplbq \afterimp \; \bQ\afterimp +\fdaeimplbv\afterimp \; \bV\afterimp +\fdaeimplbrho\afterimp.
\end{equation}
The corresponding sensitivity jump matrix \eqref{eqn:Jump_direct_sensitivity_constraints} is:
\begin{eqnarray}
\label{eqn:Jump_direct_sensitivity_constraints2}
\begin{bmatrix}
\bQ_\textnormal{dep+}\afterimp\\
\bQ_\textnormal{dof+}\afterimp \\  
\bV\afterimp \\
\Lambda\afterimp \\
\Gamma\afterimp\\ 
\bZ\afterimp
\end{bmatrix}
=
\underbrace{
	\begin{bmatrix}
	\left(\bQ_\textnormal{dep+}\afterimp \right)_{\bQ\beforeimp }   & \Zero_{(n-f)\times n}  & \Zero_{(n-f)\times m}  &  D& \Zero_{(n-f)\times \nc} \\
	\left(\bQ_\textnormal{dof+}\afterimp \right)_{\bQ\beforeimp}   & \Zero_{f\times n}  & \Zero_{f\times m}  &  \Zero_{f\times p}   & \Zero_{f\times \nc} \\
	\fdaeimpvbq \afterimp \; & \fdaeimpvbv\afterimp \; & \Zero_{n \times m} & \fdaeimpvbrho\afterimp & \Zero_{n \times nc} \\	
	\fdaeimplbq \afterimp \; & \fdaeimplbv\afterimp \; & \Zero_{m \times m} & \fdaeimplbrho\afterimp & \Zero_{m \times nc} \\	
	\Zero_{p\times n}  & \Zero_{p\times n}   & \Zero_{p\times m}  &  \bI_{p\times p}   & \Zero_{p\times \nc} \\
	\left(\bZ\afterimp\right)_{\bQ\beforeimp} & \Zero_{\nc \times n}  & \Zero_{\nc \times m} & 
	\Zero_{\nc \times p}  & \bI_{\nc \times \nc}  
	\end{bmatrix} 
}_{\mathsf{S}_\textnormal{eve}}
\cdot
\begin{bmatrix}
\bQ_\textnormal{dep+}\beforeimp\\
\bQ_\textnormal{dof+}\beforeimp \\
\bV\beforeimp\\
\Lambda\beforeimp  \\
\Gamma\beforeimp \\  
\bZ\beforeimp \\
\end{bmatrix}.
\end{eqnarray}
\end{definition}

\subsection{Adjoint sensitivity analysis for hybrid constrained dynamical systems}
\label{sec:multibody-constrained-adjoint}

\begin{definition}[Jump in adjoint sensitivity for constrained systems with elastic impact]
The transpose of the direct sensitivity jump matrix  $\mathsf{S}_\textnormal{eve}^\textrm{T}$ \eqref{eqn:adjoint-sensitivity-jump-matrix} associated with an elastic impact \eqref{eqn:Jump_direct_sensitivity_constraints} is:
\begin{eqnarray}
\mathsf{S}_\textnormal{eve}^\textrm{T}
=
\begin{bmatrix}
\left(\bQ_\textnormal{dep+}\afterimp \right)^{\rm T}_{\bQ\beforeimp } 
&
\left(\bQ_\textnormal{dof+}\afterimp \right)^{\rm T}_{\bQ\beforeimp}   
&
 \left(\bV_\textnormal{dep+}\afterimp \right)^{\rm T}_{\bQ\beforeimp }
&
 \left(\bV_\textnormal{dof+}\afterimp \right)^{\rm T}_{\bQ\beforeimp}  
&
  \Zero_{n \times p}                                    
&
 \left(\bZ\afterimp\right)^{\rm T}_{\bQ\beforeimp} 
\\
\Zero_{(n-f) \times (n-f)}  & \Zero_{(n-f) \times f}  
& \Zero_{(n-f) \times (n-f)} &  \Zero_{(n-f) \times f} &
 0 _{(n-f) \times p} & 0 _{(-n-f) \times \nc} 
\\
\Zero_{f \times (n-f)}  & \Zero_{f \times f}  
& \left(\bV_\textnormal{dep+}\afterimp \right)^{\rm T}_{\bV_\textnormal{dof+}\beforeimp } & \left[ h_\bv\beforeimp \right]^{\rm T}
& 0 _{f \times p} & 0 _{f \times p} 
\\
D^{\rm T}  & \Zero_{p \times f} &
K^{\rm T}  & \left[ h_\brho\beforeimp \right]^{\rm T}&
\bI_{p \times p}  &0 _{p \times \nc}
\\
 0 _{\nc \times (n-f)} & 0 _{\nc \times f} &
0 _{\nc \times (n-f)}  &\Zero_{\nc \times f}
& 0 _{\nc \times p} & \bI_{\nc \times \nc}
\end{bmatrix}
\end{eqnarray}
From the adjoint sensitivity equation \eqref{eqn:adjoint-sensitivity-jump-matrix} the jumps in adjoint variables for ODE systems with constraints undergoing an elastic impact are: 
\begin{equation}
\begin{split}
\lambda^\bQ\beforeimp
= &
\left[
 \left(\bQ_\textnormal{dep+}\afterimp \right)^{\rm T}_{\bQ\beforeimp } \;
\left(\bQ_\textnormal{dof+}\afterimp \right)^{\rm T} _{\bQ\beforeimp} \right] \cdot  \lambda^\bQ\afterimp  +
 \left[ \left(\bV_\textnormal{dep+}\afterimp \right)^{\rm T}_{\bQ\beforeimp }  \;
  \left(\bV_\textnormal{dof+}\afterimp \right)^{\rm T}_{\bQ\beforeimp}  \right]
  \cdot  	\lambda^\bV\afterimp +
\left(\bZ\afterimp\right)_{\bQ\beforeimp}^{\rm T}  \cdot  {\lambda}^{\bZ}\afterimp
\\
\lambda^\bV\beforeimp =& \;
\begin{bmatrix}
\Zero_{(n-f) \times f} &  \Zero_{(n-f) \times f} \\
\left(\bV_\textnormal{dep+}\afterimp \right)^{\rm T}_{\bV_\textnormal{dof+}\beforeimp }  
 & \left[h_\bv\beforeimp \right]^{\rm T}
\end{bmatrix} 
\cdot \lambda^\bV\afterimp 
\\
{\lambda}^{\Gamma}\beforeimp =& \; \begin{bmatrix}
D^{\rm T}  & \Zero_{p \times f}  
\end{bmatrix} \cdot  \lambda^\bQ\afterimp
+
\begin{bmatrix}
K^{\rm T}  &
{h_{\brho}\beforeimp}^{\rm T}  
\end{bmatrix} \cdot  \lambda^\bV\afterimp +
{\lambda}^{\Gamma}\afterimp
\\
{\lambda}^{\bZ}\beforeimp = & \; {\lambda}^{\bZ}\afterimp
\end{split}
\end{equation}
\end{definition}

\begin{definition}[Jump in adjoint sensitivity for constrained systems with inelastic impact and a sudden change of constraints]
The transpose of the direct sensitivity jump matrix  $\mathsf{S}_\textnormal{eve}^\textrm{T}$ \eqref{eqn:adjoint-sensitivity-jump-matrix} associated with \eqref{eqn:Jump_direct_sensitivity_constraints2} is:
	\begin{eqnarray}
	\mathsf{S}_\textnormal{eve}^\textrm{T}
	=
	\begin{bmatrix}
	\left(\bQ_\textnormal{dep+}\afterimp \right)^{\rm T}_{\bQ\beforeimp }  & 
	\left(\bQ_\textnormal{dof+}\afterimp \right)^{\rm T}_{\bQ\beforeimp}  
	&
	\left[\fdaeimpvbq \afterimp \right]^{\rm T}  &  \left[\fdaeimplbq \afterimp \right]^{\rm T}
	&
	\Zero_{n \times p} &
	\left(\bZ\afterimp\right)^{\rm T}_{\bQ\beforeimp}
	\\
	\Zero_{n \times (n-f)}	& \Zero_{n \times f} &
	\left[\fdaeimpvbv \afterimp \right]^{\rm T}  &  \left[\fdaeimplbv \afterimp \right]^{\rm T}
	& 	\Zero_{n \times (p)} & 	\Zero_{n \times (\nc)}
	\\
	\Zero_{m \times (n-f)}	& \Zero_{m \times f}	&
	\Zero_{m \times n} &  \Zero_{m \times m}
	&	\Zero_{m \times (p)}&	\Zero_{m \times (\nc)}
	\\
	D^{\rm T}  & \Zero_{p \times f} & \left[\fdaeimpvbrho \afterimp \right]^{\rm T}
	&\left[\fdaeimplbrho \afterimp \right]^{\rm T} &	\bI_{p \times p}  &0 _{p \times \nc} 
	\\
	\Zero_{\nc \times (n-f)}	& \Zero_{\nc \times f}	&
	\Zero_{\nc \times n} &  \Zero_{\nc \times m}
	&	\Zero_{\nc \times (p)} & \bI_{\nc \times \nc}
	\end{bmatrix}
	\end{eqnarray}
Since the adjoints of the algebraic Lagrange variables are zero
(Remark \ref{adjoint of the algebraic variables}), the adjoint sensitivity equations \eqref{eqn:adjoint-sensitivity-jump-matrix} provide the jump equations for adjoint variables at the time of event: %
	\begin{equation}
	\begin{split}
	\lambda^\bQ\beforeimp
	= & \left[
	\left(\bQ_\textnormal{dep+}\afterimp \right)^{\rm T}_{\bQ\beforeimp }  \,
	\left(\bQ_\textnormal{dof+}\afterimp \right)_{\bQ\beforeimp}  \right]
	\cdot  \lambda^\bQ\afterimp  +
		\left[	\fdaeimpvbq \afterimp \right]^{\rm T}   \cdot  	\lambda^\bV\afterimp +
	\left(\bZ\afterimp\right)_{\bQ\beforeimp}^{\rm T}  \cdot  {\lambda}^{\bZ}\afterimp,
	\\
	\lambda^\bV\beforeimp =& \;
		\left[	\fdaeimpvbv \afterimp \right]^{\rm T} 
	\cdot \lambda^\bV\afterimp, 
	\\
	{\lambda}^{\Gamma}\beforeimp =& \; \begin{bmatrix}
	D^{\rm T}  & \Zero_{p \times f}  
	\end{bmatrix} \cdot  \lambda^\bQ\afterimp
	+
		\left[	\fdaeimpvbrho \afterimp \right]^{\rm T}   \cdot  \lambda^\bV\afterimp +
	{\lambda}^{\Gamma}\afterimp,
	\\
	{\lambda}^{\bZ}\beforeimp = & \; {\lambda}^{\bZ}\afterimp.
	\end{split}
	\end{equation}
\end{definition}

\begin{remark}[Sensitivities of the cost function]
Once the evolution of the sensitivities of the direct or adjoint sensitivities are computed, the sensitivity of the cost function with respect to parameters $d\,\psi/d\, \brho$ is obtained from equations \eqref{eqn:sensitivity-psi-direct} and \eqref{eqn:sensitivity-psi-adjoint}. Note that the evolution of the direct and adjoint sensitivities involve is piecewise continuous in time, with jumps occurring at each event.
\end{remark}
\section{Case study: sensitivity analysis of a five-bar mechanism}
\label{sec:numerics}
The five-bar mechanism, presented in Fig.~\ref{Fig_fivebar}, is used as a case study to validate the adjoint sensitivity method in computing the sensitivity of cost functions with respect to parameters for hybrid constrained dynamical systems.
The mechanism has two degrees of freedom, five revolute joints located at points A, 1, 2, 3, and B; the masses of each bars are $m_1=1 \; kg$, $m_2=1.5 \; kg$, $m_3=1.5 \; kg$, $m_4=1 \; kg$; the polar moments of inertia for each bars have uniformly distributed mass; the two springs have stiffness coefficients of $k_1=k_2=100 \; N/m$ and natural lengths of $L_{01}=2.2360 \; m$ and $L_{02}=2.0615 \; m$.
The state vector $\bq= \big[  \bq^{\rm T}_{1} ~ \bq^{\rm T}_{2} ~ \bq^{\rm T}_{3} \big]^{\rm T}$ includes the natural coordinates of the point 1, 2, and 3 of the mechanism. The coordinates  $\bq_{2}= \big[ x_{2} ~ y_{2}  \big]^{\rm T}$ are independent and defines  the DOF of the system, while the coordinates $\bq_{1}= \big[  x_{1} ~ y_{1}  \big]^{\rm T}$ with  $\bq_{3}= \big[ x_{3} ~ y_{3} \big]^{\rm T}$ are dependent. The constraint equations, used to solve for the dependent coordinates, are defined according to the fixed lengths between each set of points, as follows:
\begin{align}
	\label{eq:Constraints 5 bar}
    {\bPhi} &= \begin{bmatrix}
           \|{q_A-q_1}\|^2  -L_{A1}^2\\
			\|{q_2-q_1}\|^2  -L_{21}^2\\
           \|{q_3-q_2}\|^2  -L_{32}^2\\
			\|{q_B-q_3}\|^2  -L_{B3}^2\\
         \end{bmatrix}=0,
\end{align}
with the lengths $L_{A1}=L_{B3}=1.4142 \; m$ ; $L_{A1}=L_{B3}=1.8027 \; m $ and the ground points $\bq_{A}= \big[ \begin{array}{c c} -0.5 & 0 \end{array} \big]^{\rm T}$; $\bq_{B}= \big[ \begin{array}{c c} 0.5 & 0 \end{array} \big]^{\rm T}$. 
The study focuses on  point 2 of the five-bar mechanism. This points  hits the ground at -2.35 m along the vertical y axis which is detected by an the event function $r(\cdot)$ described in Eq.~\eqref{eq:event_Function}. At the time of event, the vertical velocity of point 2 jumps to its opposite value, while its horizontal velocity remains the same. 
The ODE forward system, the direct and adjoint sensitivity are simulated with a time span of five seconds. The residuals of the constraint equations, presented  in Fig.~\ref{fig2}, shows that the position and the velocity constraints are satisfied within an error of $10^{-6}$ and $10^{-5}$, respectively, which is satisfactory. 
\begin{figure} [H] 
\centering
	\begin{subfigure}[t]{.49\textwidth}
	\centering
	\raisebox{4mm}{\includegraphics[width=0.9\textwidth]{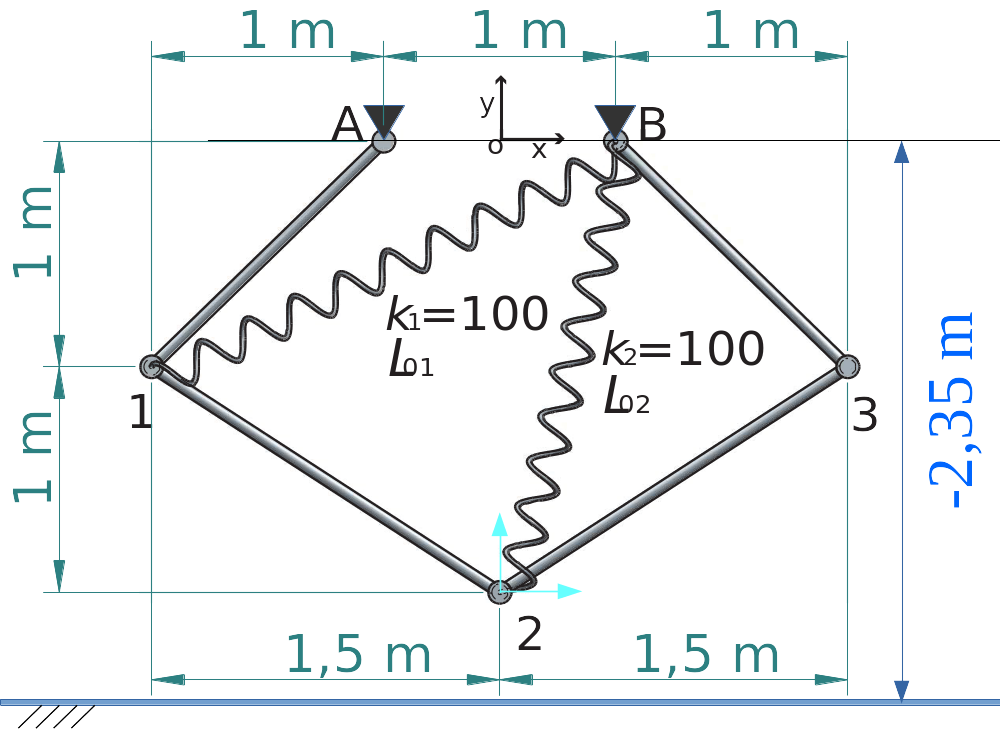}}
	\caption{Diagram of the five-bar mechanism.}
	\label{Fig_fivebar}
	\end{subfigure}
	\begin{subfigure}[t]{.49\textwidth}
	\centering
	\includegraphics[width=\textwidth]{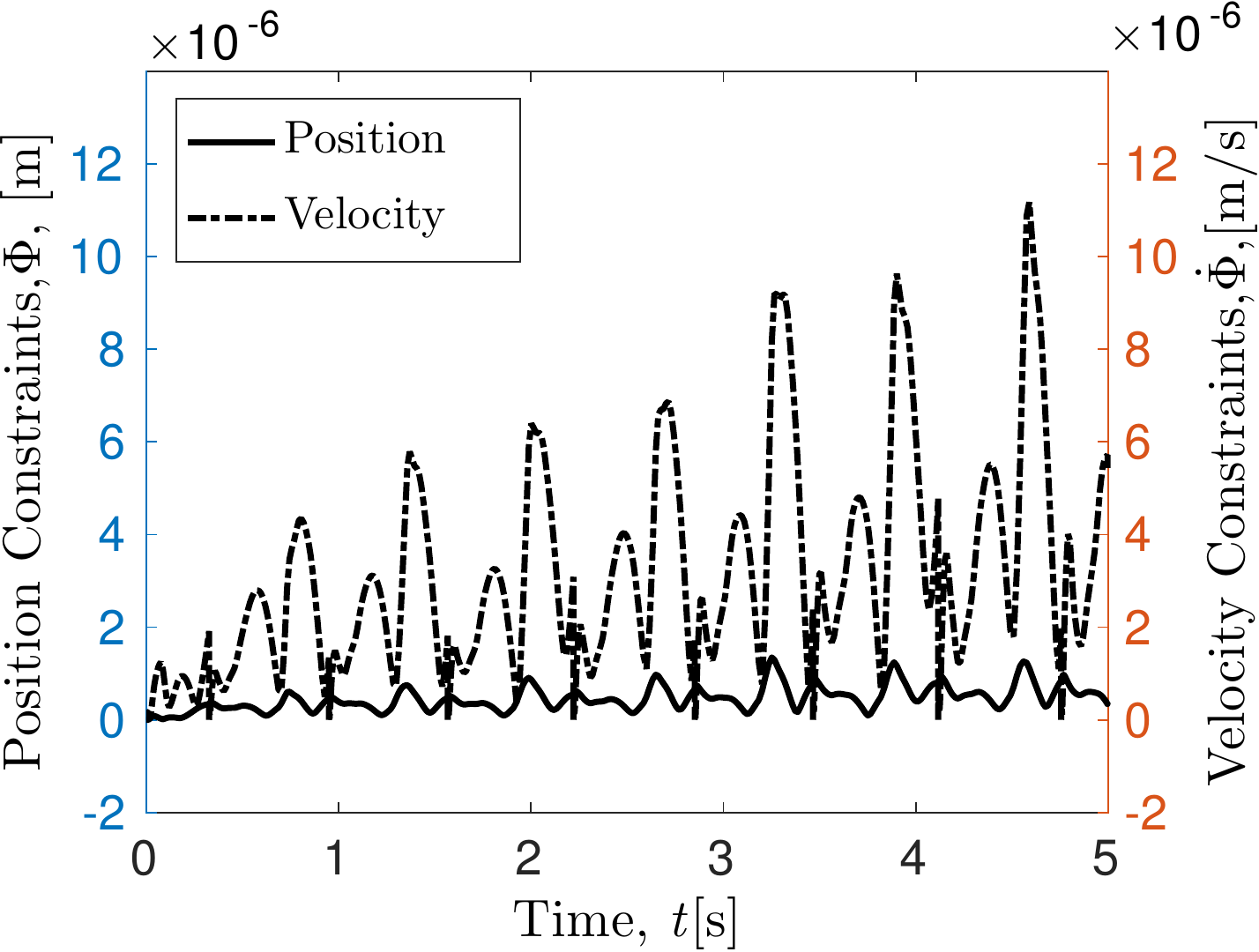}
	\captionsetup{margin=1cm}
	\caption{The position and the velocity constraint residuals for the five-bar mechanism.}
	\label{fig2}
	\end{subfigure}
\caption{Structure of the five-bar mechanism.}
\end{figure}
The trajectories of the vertical position and velocity of point 2 of the five-bar mechanism are presented in Fig.~\ref{fig3} and Fig.~\ref{fig4}, respectively. As expected, the point 2's vertical position bounces at -2.35m, and its vertical velocity jumps at each time of event with $\bv_2\afterimp=\bv_2\beforeimp$.

\begin{figure} [H] 
	\centering
	\begin{subfigure}{.49\textwidth}
		\centering
		\includegraphics[width=\textwidth]{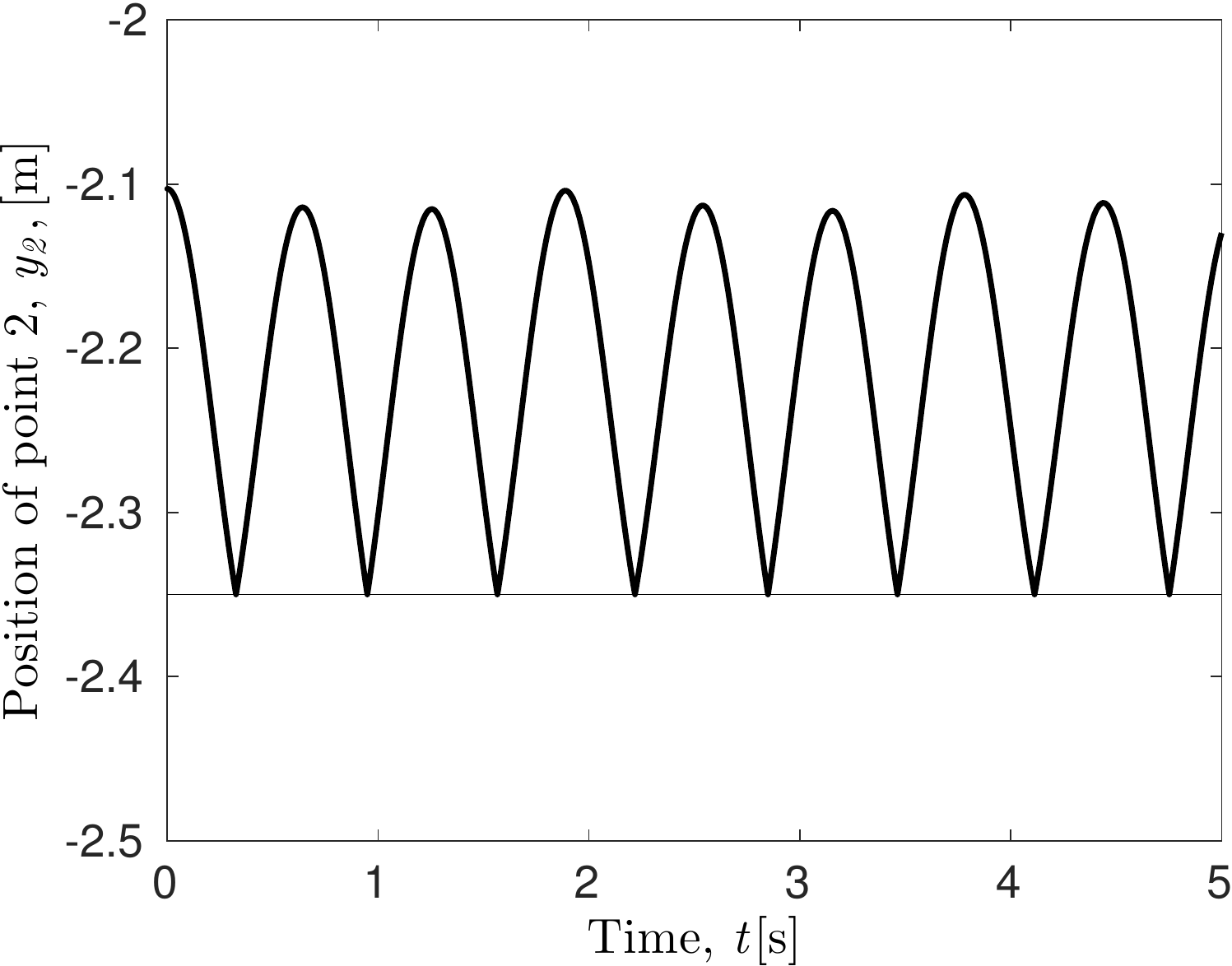}
		\captionsetup{margin=1cm}
		\caption{The position of point 2.}
		\label{fig3}
	\end{subfigure}
	\begin{subfigure}{.49\textwidth}
		\centering
		\includegraphics[width=\textwidth]{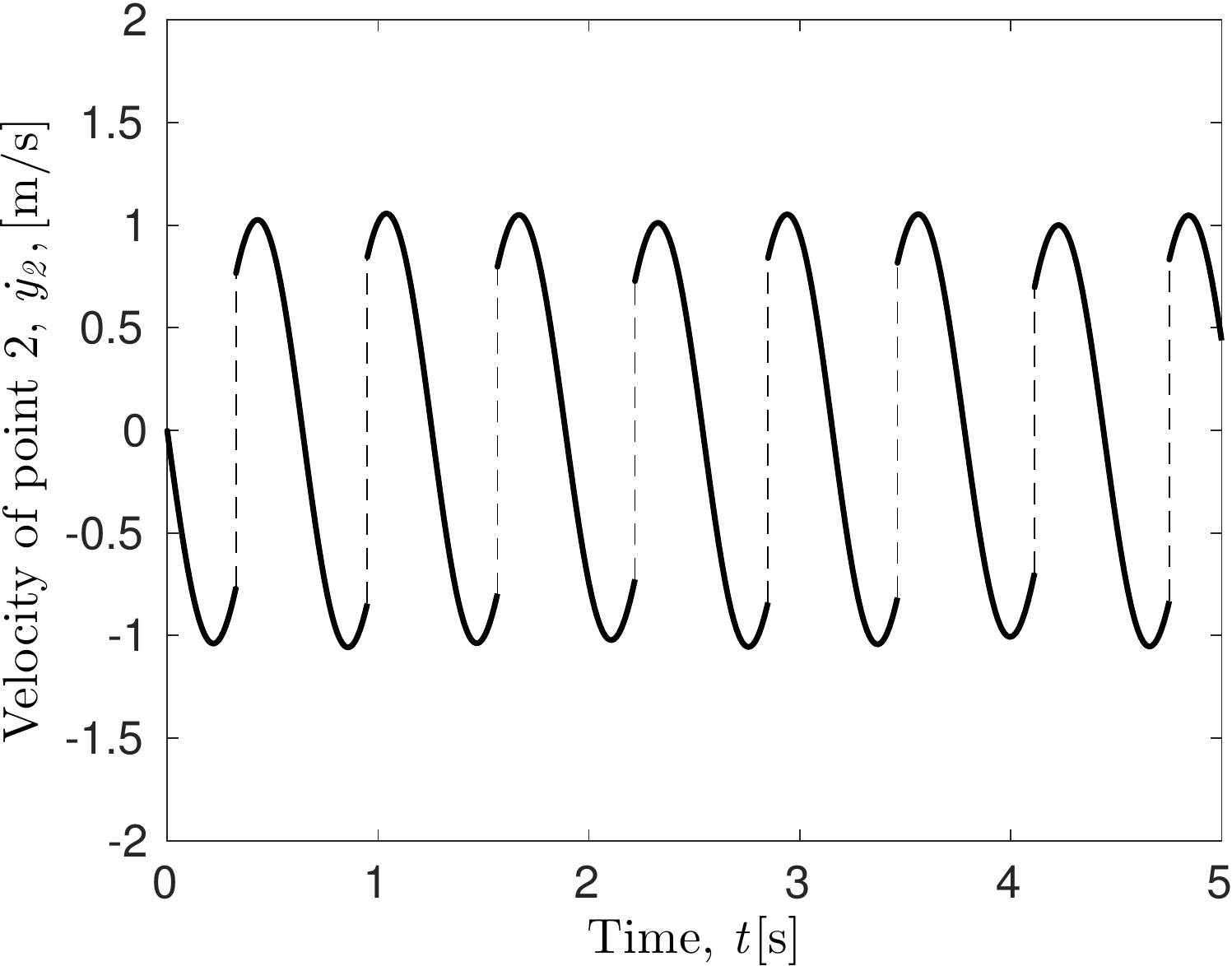}
		\captionsetup{margin=1cm}
		\caption{The velocity of point 2.}
		\label{fig4}
	\end{subfigure}
	\caption{The state variables of point 2 of the five-bar mechanism.}
	\label{sfig4}
\end{figure}

The trajectory of the quadrature variable $z(t) =\dytwo$  of the five-bar mechanism and its sensitivity are shown in Fig.~\ref{fig5} and Fig.~\ref{fig6}, respectively. The same analysis is presented in Fig.~\ref{fig7} and Fig.~\ref{fig8} for the quadrature variable $\bz(t) =\ddytwo$.
The direct sensitivity is represented by the continuous line, while the central finite difference sensitivity is represented by the dashed line. Both solutions were solved forward in time. The adjoint sensitivity is presented as well, and was solved backwards in time. As presented in our previous paper, the direct differentiation method to compute the sensitivity of the cost function with discontinuities in the velocity state variables of the mechanism is validated. The validation comes to the fact that the trajectories of the sensitivity of the quadrature variable $\Z$ exactly matches the trajectory of  numerical sensitivity computed with a finite difference method. 
One main conclusion of our previous paper was to state that our proposed direct sensitivity method in computing the sensitivity of the cost function with discontinuities in the velocity state variables was more robust than the numerical method. Indeed, the direct method  accurately determines the jump in the sensitivities  and their trajectories. This after each event, without any delta-like jumps in magnitude $1/\eps$ that occurs in the numerical method at each time of event.
This validated direct sensitivity method is now compared to the proposed adjoint method in computing the sensitivity of the cost function with discontinuities in the velocity state variables. The results  presented in  Fig.~\ref{fig5} and Fig.~\ref{fig7} show that the adjoint and direct method exactly converge to the same sensitivity cost number with a difference of less than 0.01 $\%$. This convergence in both methods validates the adjoint sensitivity method in computing the sensitivity of the cost function with discontinuities in the trajectories.
\begin{figure} [H] 
	\centering
	\begin{subfigure}{.48\textwidth}
	\centering
	\includegraphics[width=\textwidth]{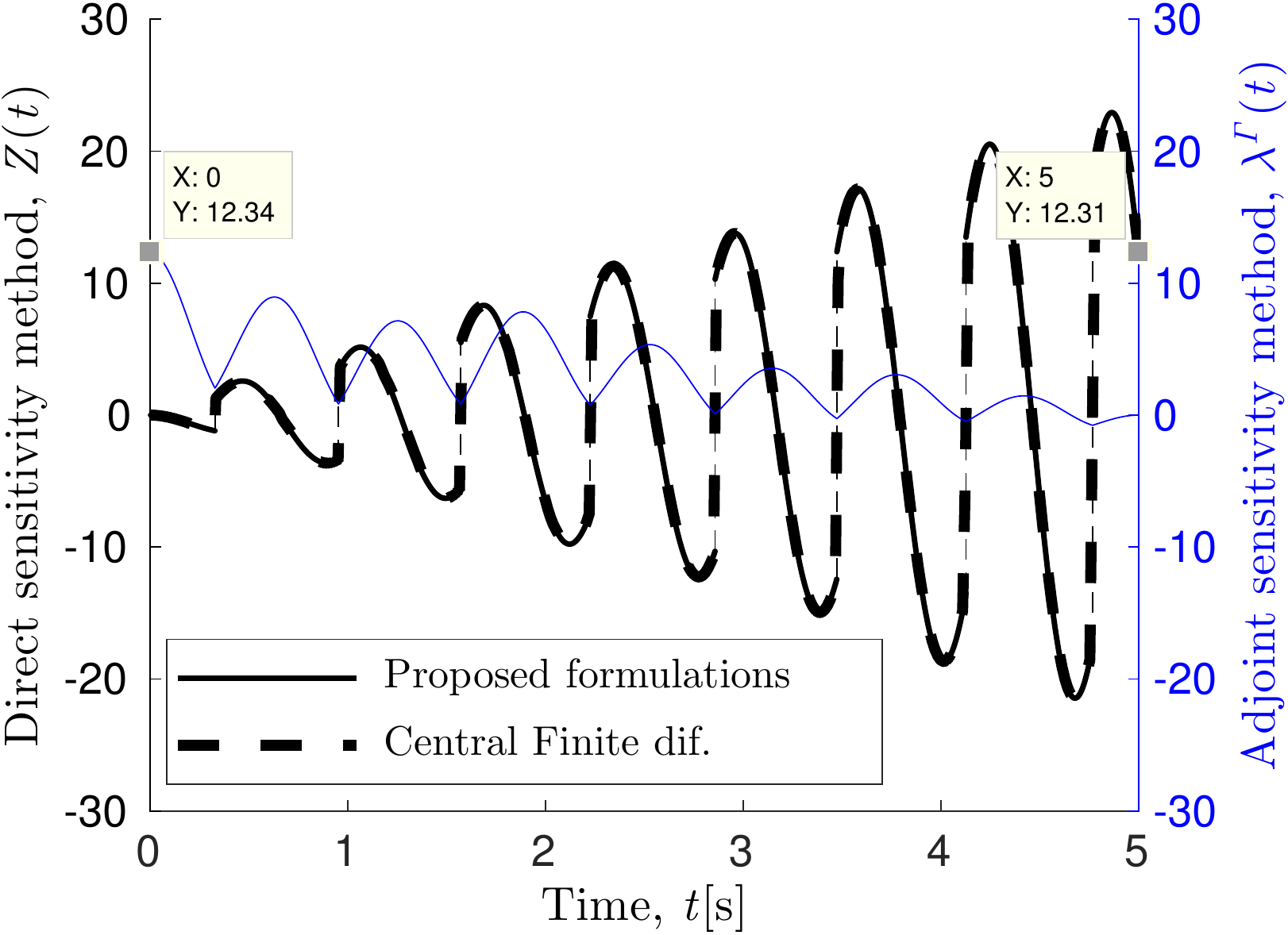}
	\captionsetup{margin=1cm}
	\caption{Direct and adjoint sensitivities.}
	\label{fig5}
	\end{subfigure}
	$\qquad$
	\begin{subfigure}{.46\textwidth}
		\centering
		\includegraphics[width=\textwidth]{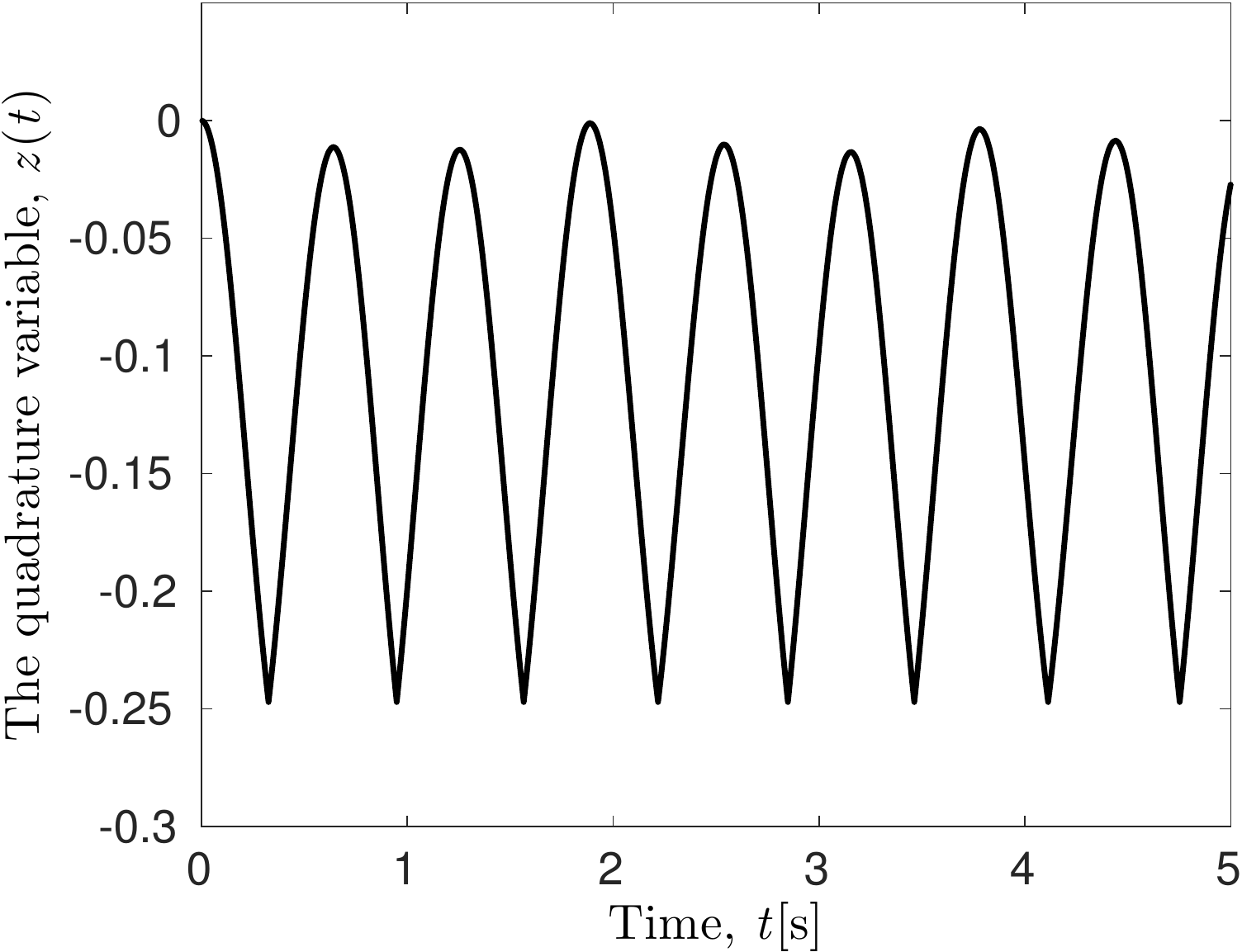}
		\captionsetup{margin=1cm}
		\caption{The quadrature variable $\bz(t)$.}
		\label{fig6}
	\end{subfigure}
	\caption{Sensitivity analysis of the five-bar mechanism with  $\bz(t) =\dytwo$ .}
\end{figure}
\begin{figure} [H] 
	\centering
	\begin{subfigure}{.48\textwidth}
		\centering
		\includegraphics[width=\textwidth]{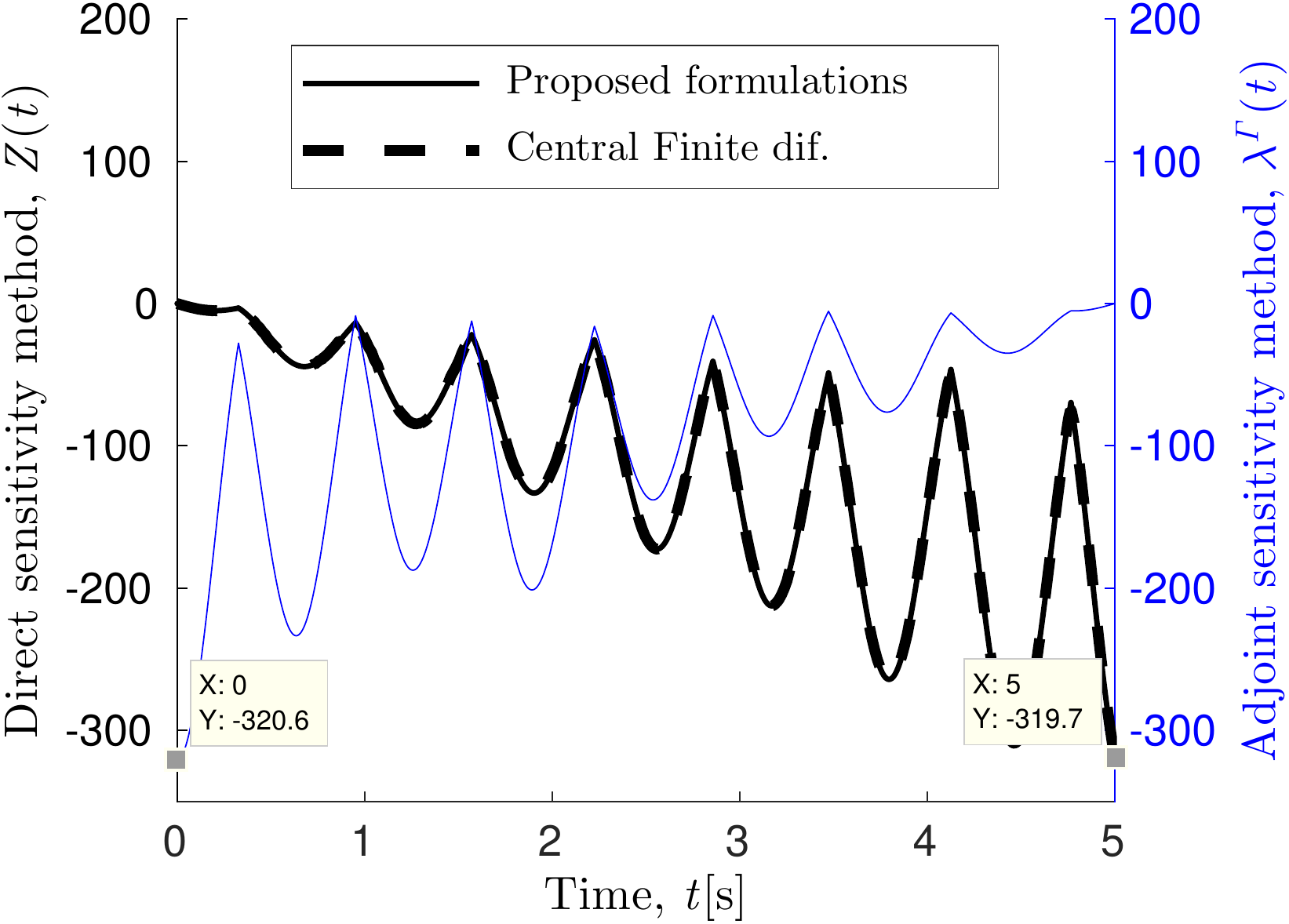}
		\captionsetup{margin=1cm}
		\caption{Direct and adjoint sensitivities.}
		\label{fig7}
	\end{subfigure}
	$\qquad$
	\begin{subfigure}{.46\textwidth}
		\centering
		\includegraphics[width=\textwidth]{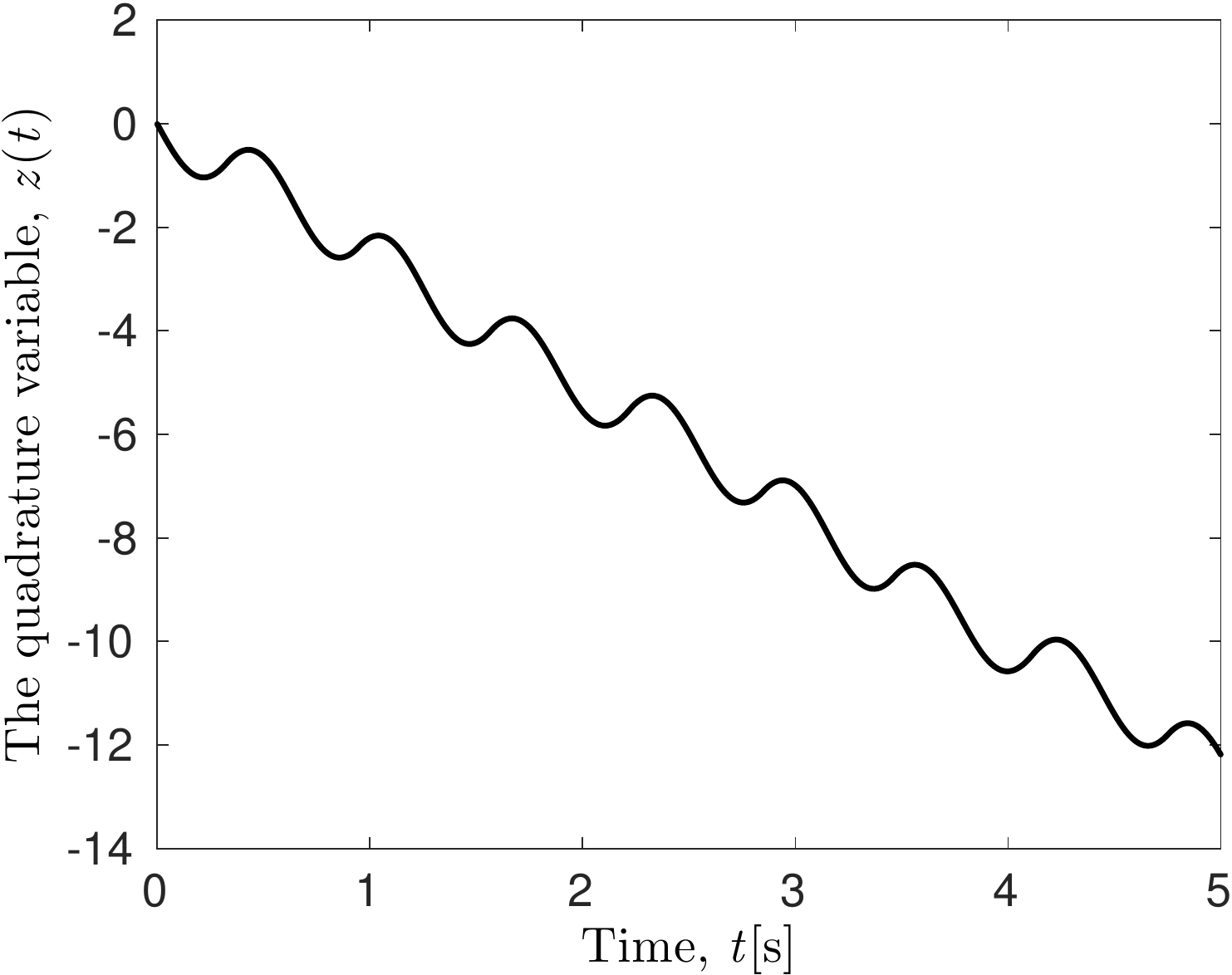}
		\captionsetup{margin=1cm}
		\caption{The quadrature variable $\bz(t)$. }
		\label{fig8}
	\end{subfigure}
	\caption{Sensitivity analysis of the five-bar mechanism with  $\bz(t) =\ddytwo$ .}
\end{figure}
Note that  $\zddytwo$ does not completely match the trajectory of the velocity of point 2 in Fig.~\ref{fig4}. Indeed, the point2's velocity jumps at the time of event, while the quadrature variable does not. The quadrature variable evaluates the integral of the acceleration of point 2 only.
\newline
The same analysis is provided with the  quadrature variable $\bz(t) =\int_{t_0}^{t} {\ddot{y}_2(\tau)}^2  +{\dot{y}_2(\tau)}^2 \ {\rm d\tau}$ in Fig.~\ref{fig9} and Fig.~\ref{fig10}. The adjoint and direct method  converge to the same sensitivity cost number with a difference of less than 0.01 $\%$.

\begin{figure} [H] 
	\centering
	\begin{subfigure}{.49\textwidth}
		\centering
		\includegraphics[width=\textwidth]{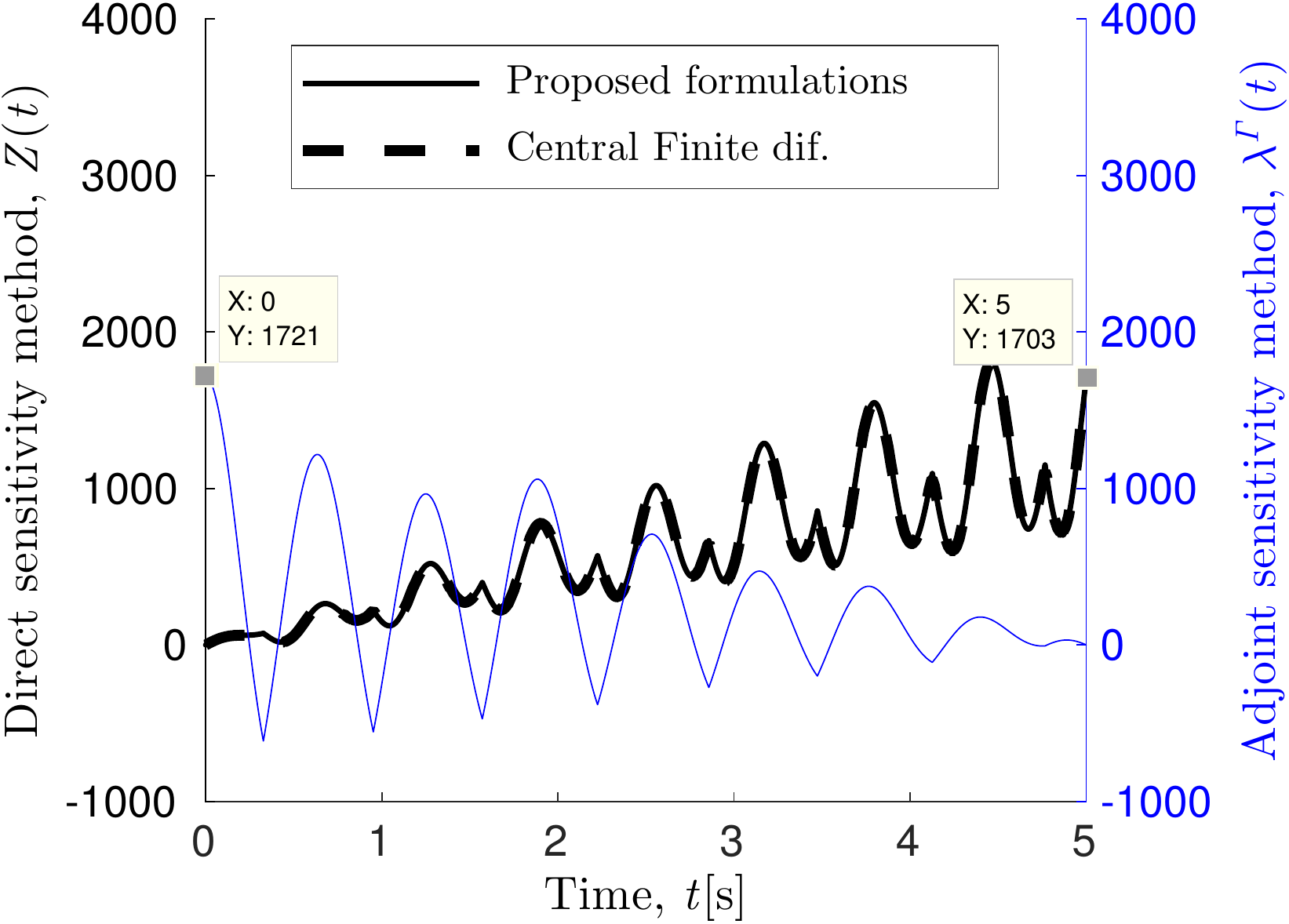}
		\captionsetup{margin=1cm}
		\caption{Direct and adjoint sensitivities.}
		\label{fig9}
	\end{subfigure}
	$\qquad$
	\begin{subfigure}{.45\textwidth}
		\centering
		\includegraphics[width=\textwidth]{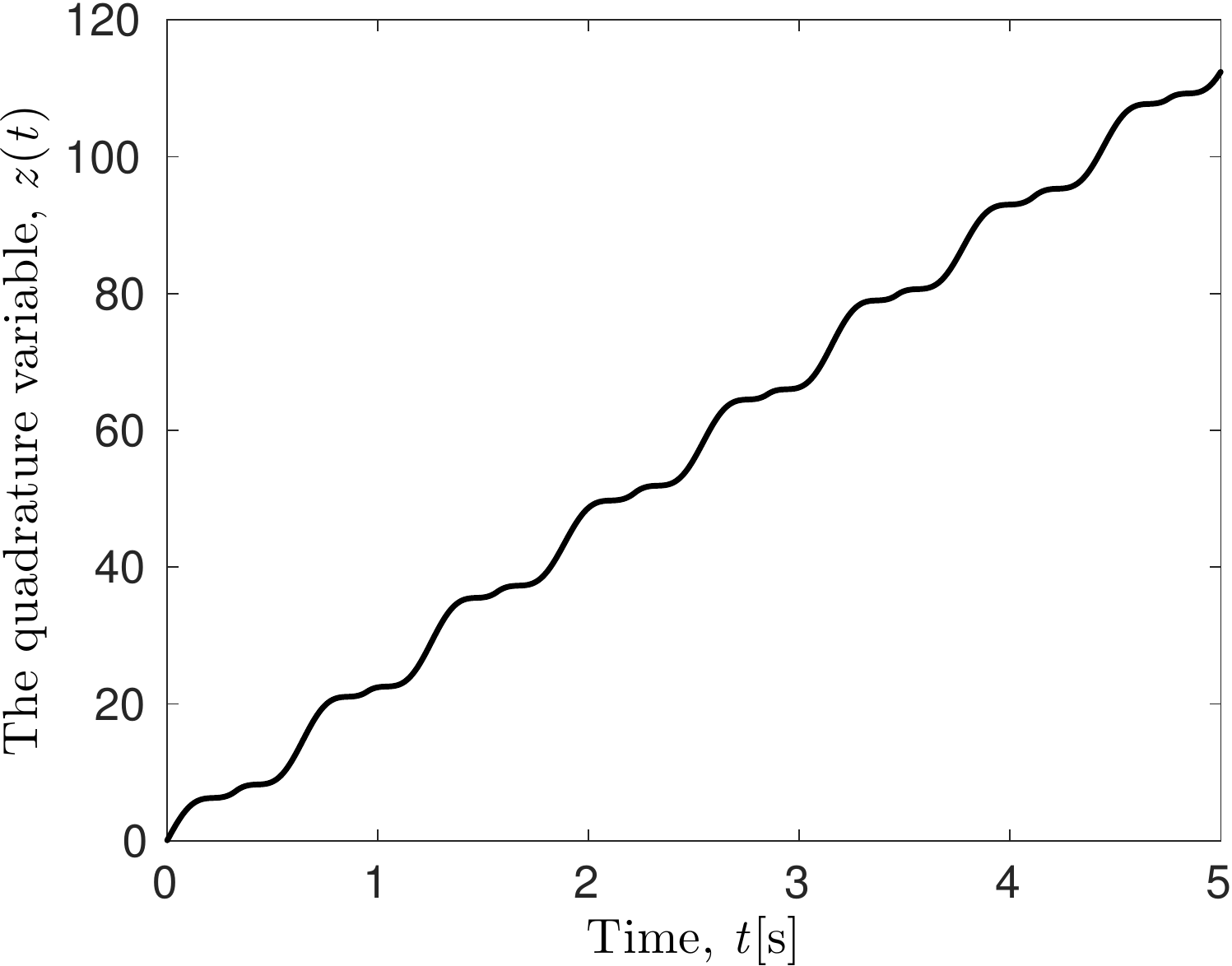}
		\captionsetup{margin=1cm}
		\caption{The quadrature variable $\bz(t)$.}
		\label{fig10}
	\end{subfigure}
	\caption{Sensitivity analysis of the five-bar mechanism with  $\bz(t) =\int_{t_0}^{t} {\ddot{y}_2(\tau)}^2  +{\dot{y}_2(\tau)}^2 \ {\rm d\tau}$  .}
\end{figure}
\section{Conclusions}
\label{sec:conclusions}
Gradient based algorithms are widely used in computational engineering problems such as design and control optimization, implicit time integration methods, and deep learning. Sensitivity analysis plays a key role in this type of algorithms as it provides the necessary derivative information. In the context of dynamical systems governed by ordinary or differential algebraic equations, sensitivity analysis computes the derivatives of general cost functions that depend on the system solution with respect to parameters or initial conditions. 

Direct and adjoint sensitivity analyses for continuous multibody dynamic systems  have been discussed in the literature \cite{Sandu_2014_sensitivity_ODE_multibody,Sandu_2014_MBSVT,Zhu_2014_PhD,Sandu2015dynamic,Sandu_2017_vehicleOptimization}.  Our earlier work has extended the direct sensitivity analysis to hybrid multibody dynamic systems systems that are subject to events such as impacts or sudden changes in constraints \cite{Corner2017}. 

This paper extends the mathematical framework to compute adjoint sensitivities for hybrid multibody dynamic systems modeled by ordinary differential equations and by index-1 differential algebraic equations. A very general formulation of the cost functions is used. For the hybrid systems considered herein discontinuities in the forward trajectories appear at time moments triggered by an event. Jump conditions for adjoint sensitivity variables are provided for mechanical systems with and without constraints. These jump conditions handle the change in the sensitivities caused by the non-smoothness of the forward trajectories at a finite number of events.

We validate the mathematical framework for adjoint sensitivities for hybrid dynamical systems on the study of a five-bar mechanism with non-smooth contacts. The direct and adjoint sensitivities computed by the proposed mathematical framework are validated against numerical sensitivities calculated by real finite differences. The results of this study show that all the alternative analyses provide the same sensitivities of the general cost function with respect to model parameters, within an error of 0.01$\%$.

Future work will extend the mathematical framework to calculate adjoint sensitivities of hybrid mechanical  systems with respect to actuation functions. These sensitivities will allow to solve optimal control problems for hybrid systems.

\appendix
\renewcommand*{\thesection}{\Alph{section}}
\section{Calculation of partial derivatives used in sensitivity analyses}
\label{appendix:derivatives-calculation}

\begin{remark}
	The expressions $\feomq$, $\feomv$, and $\feomrhoi$ denote the partial derivatives of $\feom$ with respect to the subscripted variables.
	The partial derivatives $\partial \feom/\partial \zeta$ are obtained by differentiating  $ \feom$ with respect to $\zeta \in \{ \bq,\bv,\brho \}$:
	\begin{equation}
	\label{eq:DIF-EOM}
	\frac{\partial \feom}{\partial \zeta} 
	=\frac{\partial (\Mass^{-1}  \, \Force)}{\partial \zeta} 
	= -\Mass^{-1}\, \Mass_\zeta\, \Mass^{-1}\,\Force
	+ \Mass^{-1}\,\Force_\zeta 
	= \Mass^{-1}\,\left( \Force_\zeta - \Mass_\zeta\,\feom \right)
	= \Mass^{-1}\,\left( \Force_\zeta - \Mass_\zeta\,\dbv \right).
	\end{equation}
	
\end{remark}

\begin{remark}
	The expressions $\tilde{\bg}_q$, $\tilde{\bg}_v$, and $\tilde{\bg}_{\brho_i}$ denote the partial derivatives of $ \tilde{\bg}$ with respect to the subscripted variables.
	The partial derivatives $\partial  \tilde{\bg}/\partial \zeta$ are obtained by differentiating \eqref{eq:EOM-ODE} with respect to $\zeta \in \{ \bq,\bv,\brho \}$:
	\begin{align}
    \tilde{\bg}_{\zeta} &=
	g_{\zeta} +  g_\dbv\,\feomzeta + g_{\tilde{\bu}} \, \tilde{\bu}_{\zeta} \\ \nonumber &=
	g_{\zeta} +  g_\dbv\,\feomzeta + g_{{\bu}} \, {\bu}_{\zeta} \, + g_{{\bu}} \, u_\dbv\,\feomzeta,
	\end{align}
	which leads to:
	\begin{align}
	 \bigl[\ \tilde{\bg}_q \ \bQ_i +
	\tilde{\bg}_v \ \bV_i +
	\tilde{\bg}_{\rho_i} \  \bigr]_{i=1,\dots,p} & =
	 \bigl[ \big(g_\bq +  g_\dbv\,\feomq + g_{{\bu}} \, {\bu}_{\bq} \, + g_{{\bu}} \, u_\dbv\,\feomq \big)\cdot \bQ_{i}   \nonumber
	\\ & \quad +  \big(g_\bv +  g_\dbv\,\feomv + g_{{\bu}} \, {\bu}_{\bv} \, + g_{{\bu}} \, u_\dbv\,\feomv \big)\cdot \bV_{i} 
	\\  \nonumber & \quad + g_{\brho_i} +  g_\dbv\cdot\feomrhoi + g_{{\bu}} \, {\bu}_{\brho} \, + g_{{\bu}} \, u_\dbv\,\feomrhoi \bigr]_{i=1,\dots,p}.
	\end{align}
\end{remark}

\begin{remark}
Similarly, the expressions $\tilde{w}_q$, $\tilde{w}_v$, and $\tilde{w}_{\brho_i}$ denote the partial derivatives of $ \tilde{w}$ with respect to the subscripted variables.
The partial derivatives $\partial  \tilde{w}/\partial \zeta$ are obtained by differentiating  $w$ with respect to $\zeta \in \{ \bq,\bv,\brho \}$:
\begin{align}
\tilde{w}_{\zeta} &=
w_{\zeta} +  w_\dbv\,\feomzeta + w_{\tilde{\bu}} \, \tilde{\bu}_{\zeta} \\ \nonumber &=
w_{\zeta} +  w_\dbv\,\feomzeta + w_{{\bu}} \, {\bu}_{\zeta} \, + w_{{\bu}} \, w_\dbv\,\feomzeta.
\end{align}
\end{remark} 
\section{Adjoint of the algebraic Lagrangian coefficient}
\label{appendix:Adjoint of the algebraic Lagrangian coefficient}
%
Methods to compute the adjoint of an index-1 DAE available in the literature \cite{Dopico2014,Ballard2000,Schaffer2005} use the following approach. 
Define the Lagrangian using the multipliers $\mu^Q , \mu^V , \mu^\Gamma$ that correspond to the constraints posed by the index-1 DAE equations \eqref{eq:EOM-DAE-index1}:
\begin{equation}
\label{eq:adjoint variables that constraints the DAE equations}
\begin{bmatrix}
\mu^Q \\ \mu^V \\ \mu^\Gamma
\end{bmatrix}^T
\cdot
\left(
\begin{bmatrix}
\bI & \bzero & \bzero \\
\bzero & {\Mass}\left(t,\bq,\brho\right) & \dPhidq^{\rm T}\left(t,\bq,\brho\right) \\
\bzero & \dPhidq\left(t,\bq,\brho\right) & \bzero
\end{bmatrix}
\cdot
\begin{bmatrix}
\dbq \\ \dbv \\ \mu
\end{bmatrix}
-
\begin{bmatrix}
\bv \\
{\Force} \left(t,\bq,\bv,\brho\right)  \\
\Faccel\left(t,\bq,\bv,\brho\right)
\end{bmatrix}
\right).
\end{equation}
We rearrange equation  \eqref{eq:adjoint variables that constraints the DAE equations} as follows: 
\begin{equation}
\left(
\begin{bmatrix}
\mu^Q \\ \mu^V \\ \mu^\Gamma
\end{bmatrix}
\cdot
\begin{bmatrix}
\bI & \bzero & \bzero \\
\bzero & {\Mass}\left(t,\bq,\brho\right) & \dPhidq^{\rm T}\left(t,\bq,\brho\right) \\
\bzero & \dPhidq\left(t,\bq,\brho\right) & \bzero
\end{bmatrix}
\right)
\cdot
\left(
\begin{bmatrix}
\dbq \\ \dbv \\ \mu
\end{bmatrix}
-
\begin{bmatrix}
\bI & \bzero & \bzero \\
\bzero & {\Mass}\left(t,\bq,\brho\right) & \dPhidq^{\rm T}\left(t,\bq,\brho\right) \\
\bzero & \dPhidq\left(t,\bq,\brho\right) & \bzero
\end{bmatrix}^{-1}
\cdot
\begin{bmatrix}
\bv \\
{\Force} \left(t,\bq,\bv,\brho\right)  \\
\Faccel\left(t,\bq,\bv,\brho\right)
\end{bmatrix}
\right)
\end{equation}

\begin{equation}
=\left(
\begin{bmatrix}
\mu^Q \\ \mu^V \\ \mu^\Gamma
\end{bmatrix}^T
\cdot
\begin{bmatrix}
\bI & \bzero & \bzero \\
\bzero & {\Mass}\left(t,\bq,\brho\right) & \dPhidq^{\rm T}\left(t,\bq,\brho\right) \\
\bzero & \dPhidq\left(t,\bq,\brho\right) & \bzero
\end{bmatrix}
\right)
\cdot
\left(
\begin{bmatrix}
\dbq \\ \dbv \\ \mu
\end{bmatrix}
-
\begin{bmatrix}
V \\ \fdaedv \\ \fdaelb
\end{bmatrix}
\right)
\end{equation}

\begin{equation}
=\begin{bmatrix}
\lambda^Q \\ \lambda^V \\ \lambda^\Gamma
\end{bmatrix}^T
\cdot
\left(
\begin{bmatrix}
\dbq \\ \dbv \\ \mu
\end{bmatrix}
-
\begin{bmatrix}
V \\ \fdaedv \\ \fdaelb
\end{bmatrix}
\right).
\end{equation}
The adjoint variables $\lambda^Q , \lambda^V , \lambda^\Gamma$ defined in this paper, and the adjoint variables $\mu^Q , \mu^V , \mu^\Gamma$ used in the literature \eqref{eq:adjoint variables that constraints the DAE equations}, are related by the following matrix multiplication:
\[
\begin{bmatrix}
\lambda^Q \\ \lambda^V \\ \lambda^\Gamma
\end{bmatrix}
=
\begin{bmatrix}
\bI & \bzero & \bzero \\
\bzero & {\Mass}\left(t,\bq,\brho\right) & \dPhidq^{\rm T}\left(t,\bq,\brho\right) \\
\bzero & \dPhidq\left(t,\bq,\brho\right) & \bzero
\end{bmatrix}
\cdot
\begin{bmatrix}
\mu^Q \\ \mu^V \\ \mu^\Gamma
\end{bmatrix}.
\]
The adjoint DAE equations and boundary conditions in the ``$\mu$ formulation''  \cite{Dopico2014,Ballard2000,Schaffer2005}
can be derived from the equations and boundary conditions in the ``$\lambda$ formulation'' discussed in this paper, and vice-versa.

\clearpage
\begin{spacing}{0.25}
	
\pagestyle{empty}
\printnomenclature[3.25 cm]
\end{spacing}
\clearpage
\doublespacing

\section*{Acknowledgments}
\label{sec:acknowledgments}
This project has been partially funded by the European Union Horizon 2020 Framework Program, Marie Skłodowska Curie actions, under grant agreement no. 645736, Project EVE, Innovative Engineering of Ground Vehicles with integrated Active Chassis Systems.
It was also supported in part by awards NSF DMS--1419003, NSF CCF--1613905, NSF ACI--1709727, AFOSR DDDAS 15RT1037, by the Computational Science Laboratory.

\section*{References}
\label{sec:refs}

\bibliographystyle{elsarticle-num} 
\bibliography{biblio}

\end{document}